\documentclass[11pt,a4paper]{amsart}
\usepackage{graphics}
\usepackage{amscd}
\usepackage{comment}
\usepackage{oldgerm}
\usepackage{amsmath}
\usepackage{amssymb}
\usepackage{amsthm}
\usepackage{color}
\usepackage{bm}

\usepackage[all]{xy}
\usepackage[foot]{amsaddr}

\makeatletter
\@addtoreset{equation}{section}
\makeatother

\theoremstyle{definition}
\newtheorem{defn}[equation]{Definition}
\theoremstyle{plain}
\newtheorem{thm}[equation]{Theorem}
\newtheorem{prob}[equation]{Problem}
\newtheorem{prop}[equation]{Proposition}
\newtheorem{fact}[equation]{Fact}

\newtheorem{lem}[equation]{Lemma}

\theoremstyle{remark}
\newtheorem{rem}[equation]{Remark}
\newtheorem{ex}[equation]{Example}
\newcommand{\aaa}{\"a}

\newcommand{\Z}{\mathbb{Z}}
\newcommand{\N}{\mathbb{N}}
\newcommand{\R}{\mathbb{R}}
\newcommand{\C}{\mathbb{C}}
\newcommand{\T}{\mathbb{T}}

\newcommand{\del}{\partial}

\begin{document}

\title[Deformation quantization for Lagrangian fibrations]{A new construction of strict deformation quantization for Lagrangian fiber bundles}
\author[M. Yamashita]{Mayuko Yamashita}
\address{Research Institute for Mathematical Sciences, Kyoto University, 
606-8502, Kyoto, Japan}
\email{mayuko@kurims.kyoto-u.ac.jp}
\subjclass[]{}
\maketitle

\begin{abstract}
    We give a new construction of strict deformation quantization of symplectic manifolds equipped with a proper Lagrangian fiber bundle structure, whose representation spaces are the quantum Hilbert spaces obtained by geometric quantization.
    The construction can be regarded as a ''lattice approximation of the correspondence between differential operators and principal symbols". 
    We analyze the corresponding formal deformation quantization.  
    We also investigate into relations between our construction and Berezin-Toeplitz deformation quantization. 
\end{abstract}

\tableofcontents

\section{Introduction}
In this paper, we give a new construction of strict deformation quantization of symplectic manifolds equipped with a proper Lagrangian fiber bundle structure, whose representation spaces are the quantum Hilbert spaces obtained by geometric quantization. 

\subsection{Deformation quantizations and geometric quantizations}
First we explain some background. 
Let $(X^{2n}, \omega)$ be a $2n$-dimensional symplectic manifold. 
We have the Poisson algebra structure on $C^\infty(X)$. 
In this paper we are interested in finding {\it strict deformation quantizations} for $(X, \omega)$. 
The notion of strict (or {\it $C^*$-algebraic}) deformation quantization is introduced in \cite{Rieffel1994}. 
In this paper we use the following definition. 

\begin{defn}\label{def_strict_DQ}
Given a symplectic manifold $(X, \omega)$, a {\it strict deformation quantization} consists of the following data. 
\begin{itemize}
    \item A sequence of Hilbert spaces $\{\mathcal{H}_k\}_{k \in \N}$. 
    \item A sequence $\{Q^k\}_{k \in \N}$ of adjoint-preserving linear maps $Q^k \colon C_c^\infty(X) \to \mathbb{B}(\mathcal{H}_k)$ so that for all $f, g \in C^\infty(X)$, we have
\begin{enumerate}
    \item $\|Q^k(f)\| \to \|f\|_{C^0}$ as $k \to \infty$, and
    \item $\|[Q^k(f), Q^k(g)] +\frac{\sqrt{-1}}{k}Q^k(\{f, g\})\| =O(\frac{1}{k^2})$ as $k \to \infty$. 
\end{enumerate}
\end{itemize}
  
\end{defn}
Note that there exists many variants in the definition, and the most general one uses the notion of continuous fields of $C^*$-algebras as in \cite{Rieffel1994}.  

Another formulation of deformation quantization is {\it formal deformation quantizations} defined in \cite{BFFLS1978}, where they seek for an associative unital product $*$, called a {\it star product}, on the set of formal power series $C^\infty(X)[[\hbar]]$ satisfying
\begin{align*}
    f * g = fg + O(\hbar) \mbox{ and }
    f * g - g * f =  \{f, g\}\hbar + O(\hbar^2). 
\end{align*}
Strict and formal deformation quantizations are related as follows. 
If we have a strict deformation quantization $\{Q_k\}_{k \in \N}$ for $(X, \omega)$, we expect to solve the equations
\begin{align}\label{eq_intro_star}
    Q^k(f)Q^k(g) = \sum_{j = 0}^l \left( \frac{-\sqrt{-1}}{k}\right)^j Q^k\left(\mathcal{C}_j(f, g)\right) + O\left(\frac{1}{k^{l + 1}}\right)
\end{align}
recursively in $l$, where $\mathcal{C}_j(\cdot, \cdot)$ is expected to be given by a differential operator, and get a star product $*$ by
\begin{align*}
    f * g = \sum_{j = 0}^\infty \mathcal{C}_j(f, g) \hbar^j. 
\end{align*}

On the other hand, given a symplectic manifold $(X, \omega)$, {\it geometric quantization} is a process to produce {\it quantum Hilbert spaces}, which is, physically, expected to be representation spaces for the Poisson algebra $C^\infty(X)$.
See \cite{Woodhouse1992} for details. 
The process goes as follows. 
First we fix a {\it prequantizing line bundle} $(L, \nabla)$ on $X$, which is a hermitian line bundle with unitary connection, satisfying $\nabla^2 = -\sqrt{-1}\omega$. 
To do this we need the integrality of $\omega/(2\pi)$. 
Next we choose a {\it polarization} $\mathcal{P} \subset TX \otimes \C$, which is an integrable Lagrangian subbundle of $TX \otimes \C$. 
Then, roughly speaking we define the quantum Hilbert space $\mathcal{H}_k$ for $k \in \N$ as the space of sections of $L^k$ which is parallel, with respect to the connection induced from $\nabla$, along vectors in $\mathcal{P}$. 

In this paper we are particularly interested in the polarization coming from a proper {\it Lagrangian fiber bundle} $\mu \colon X^{2n} \to B^n$ with connected fibers. 
    In this case the fibers are $n$-dimensional affine torus by Arnold-Liouville theorem \cite{Arnold1989}. 
    A point $b \in B$ is called a $k$-{\it Bohr-Sommerfeld point} if the space of fiberwise parallel sections of $(L^k, \nabla)$, denoted by $H^0(X_b; L^k)$, is nontrivial. 
    In this paper we define the quantum Hilbert space by the following.  
    \begin{align}\label{eq_def_quantum_hilb.y}
        \mathcal{H}_k = \oplus_{b \in B_k} H^0(X_b; L^k\otimes |\Lambda|^{1/2}X_b),
    \end{align}
    where $B_k$ denotes the set of $k$-Bohr-Sommerfeld points, and $|\Lambda|^{1/2}X_b := |\Lambda|^{1/2}(T^*X_b)$ is the half-density bundle equipped with the canonical flat connection. 
    Thus, we have a one-dimensional Hilbert space on each $k$-Bohr-Sommefeld point, and the quantum Hilbert space is their direct sum.  
By Arnold-Liouville theorem, we can take a local action-angle coordinate which identifies the prequantum line bundle with the standard one, $(\R^n \times T^n, {}^t\! dx \wedge d\theta, L = \underline{\C}, \nabla = d -\sqrt{-1}{}^t\!x d\theta)$. 
In this coordinate, we have $B_k = \frac{\Z^n}{k}$. 
So we regard the set $B_k$ as a ''lattice approximation" of the integral affine manifold $B$, and $\mathcal{H}_k$ in \eqref{eq_def_quantum_hilb.y} can be regarded as approximation of $L^2(B)$ as $k \to \infty$. 
This observation is the key to our construction below. 

On the other hand, another well-studied type of polarization is {\it K\aaa hler polarization}, coming from an $\omega$-compatible complex structure $J$ on $X$. 
In this case $\mathcal{P} = T_J^{0, 1}X$, and the quantum Hilbert spaces are $L^2H^0(X_J; L^k)$, the space of $L^2$-holomorphic sections on $L^k$. 

So, the following problem arises naturally. 
\begin{prob}\label{prob_intro}
Given a prequantized symplectic manifold $(X, \omega, L, \nabla)$ and a polarization $\mathcal{P}$, construct a strict deformation quantization $\{Q^k\}_k$, whose representation spaces $\{\mathcal{H}_k\}_{k \in \N}$ are those obtained by the geometric quantization. 
\end{prob}

For K\aaa hler quantization, there are natural and well-studied answer to Problem \ref{prob_intro}, namely Berezin-Toeplitz deformation quantization. 
This is given by the multiplication operator composed with the orthogonal projection onto the space of holomorphic sections (see Definition \ref{def_BTDQ.y} below). 
It was shown that these operators have the correct semiclassical behavior, in the case for compact K\aaa ler manifolds by Bordemann, Meinrenken, and Schlichenmaier \cite{BMS1994}, and for a certain class of non-compact K\aaa hler manifolds by Ma and Marinescu \cite{MaMarinescu2008}. 
Moreover, Schlichenmaier \cite{Schlichenmaier2000} has shown that this induces a star product called Berezin-Toeplitz star product. 

However, Problem \ref{prob_intro} in the case where $\mathcal{P}$ comes from Lagrangian fiber bundles seems to have not been considered. 
The purpose of this paper is to give an answer to Problem \ref{prob_intro} in the case where $\mathcal{P}$ comes from proper Lagrangian fiber bundles, and investigate into some basic properties. 
The construction can be regarded as a ''lattice approximation of the correspondence between differential operators and principal symbols", where the set $B_k$ is regarded as the lattice approximation of $B$. 

\subsection{A motivation for the construction --- The quantization of $T^*M$}
Let $M^n$ be a smooth manifold. 
As a motivation for our construction, in this subsection we recall the usual symbol map and its "inverse" map, which can be considered as a quantization procedure of the symplectic manifold $T^*M$ equipped with the vertical real polarization $T^*M \to M$. 

The principal symbol map transforms commutators of differential operators into the Poisson bracket with respect to the canonical symplectic structure on $T^*M$, 
\begin{align*}
    \sigma_{\mathrm{pr}}([D_1, D_2]) = -\sqrt{-1}\{\sigma_{\mathrm{pr}}(D_1), \sigma_{\mathrm{pr}}(D_2)\}. 
\end{align*}

Let $\sigma \in C^{\infty}(T^*\R^n)$ be a symbol of some pseudodifferential operator on $\R^n$. 
Then, the operator $P \in \Psi^*(\R^n)$ with symbol $\sigma$ is, up to smoothing operators, recovered by the following procedure. 
Identifying $P$ with its integral kernel $K \in C^{-\infty}(\R^n \times \R^n)$, we define
\begin{align}\label{eq_fourier_trans.y}
    K(y, x) = \int  e^{-\sqrt{-1}\langle y-x, \xi \rangle}\sigma (x, \xi) d\xi. 
\end{align}
In other words, the inverse map for the principal symbol map is given by the Fourier transform on each fiber of $T^*\R^n$. 
For general manifold $M$, by patching the above fiberwise Fourier transform for $T^*M$ together, from a principal symbol, we can recover the operator up to lower order.

If we regard this procedure as a quantization map of $T^*M$ induced by the Lagrangian fibration $T^*M \to M$, it is natural to generalize this to the case of proper Lagrangian fiber bundles. 
In this case, the fiberwise Fourier transform is replaced by the fiberwise Fourier expansion, producing an operator on the lattice approximation of the base manifold. 

\subsection{Outline of the paper}
This paper is organized as follows. 
In Section \ref{sec_construction.y}, we give a construction of strict deformation quantization. 
First in subsection \ref{subsec_model.y}, we give the construction for the model case $(\R^n \times T^n, {}^t\!dx \wedge d\theta)$ equipped with the Lagrangian fiber bundle structure $\mu \colon \R^n \times T^n \to \R^n$ in Definition \ref{def_rep_model.y}, and prove that it is indeed a strict deformation quantization in Theorem \ref{thm_model}. 
Next, we deal with the general case in subsection \ref{subsec_construction_general.y}. 
The construction is given in Definition \ref{def_rep_model.y}, and prove that it is a strict deformation quantization in Theorem \ref{thm_general}. 

In Section \ref{sec_star}, we analyze the formal deformation quantization induced by our strict deformation quantization. 
Since our construction can be written explicitely in action-angle coordinate locally, in principle we can solve the equation \eqref{eq_intro_star} explicitly. 
In nice cases, we show in Theorem \ref{thm_star_integrable} that the star product obtained from our strict deformation quantization conincides with the star product given by the Fedosov's construction \cite{Fedosov1994}. 
In general cases, by an explicite computation we check this coincidence up to second order term in Theorem \ref{thm_star_gen.y}. 

In Section \ref{sec_relation.y}, we explain a relation between our construction and Berezin-Toeplitz deformation quantization. 
In general, if we have a prequantized symplectic manifold $(X, \omega)$ equipped with an $\omega$-compatible complex structure as well as a Lagrangian fiber bundle structure, there is no canonical isomorphism between quantum Hilbert spaces obtained by the two polarizations. 
Here we restrict our attention to the case of $\R^n \times T^n$ (subsection \ref{subsec_RnTn_relation.y}) and abelian varieties (subsection \ref{subsec_abelian_relation}) with translation invariant complex structure. 
In those cases we have a natural isomorphism between quantum Hilbert spaces using theta basis for $L^2$-holomorphic sections on $L^k$. 
We show in Theorem \ref{prop_rel_dq.y} and Theorem \ref{prop_rel_dq_abelian.y} that, as $k \to \infty$, the operator norm of the difference between our deformation quantization and Berezin-Toeplitz deformation quantization converges to zero in both cases.

\subsection{Notations}
\begin{itemize}
    \item We let $T := \R / (2\pi \Z)$. 
    \item We denote the trivial hermitian line bundle over a manifold by $\underline{\C}$. 
    \item For a smooth manifold $M$, we denote by $|\Lambda|^{1/2}M := |\Lambda|^{1/2}(T^*M)$ its half-density bundle. 
    \item For a Hilbert space $\mathcal{H}$, we denote by $\mathbb{B}(\mathcal{H})$ its bounded operator algebra. 
    \item Given a fiber bundle $\mu \colon X \to B$ and a subset $U \subset B$, we write $X_U := \mu^{-1}(U)$. 
    For a point $b \in B$, we write $X_b := \mu^{-1}(b)$. 
    \item For a subset $U \subset B$ of the base of a Lagrangian fiber bundle, we denote by $\mathcal{H}_k|_U$ the subspace of the quantum Hilbert space $\mathcal{H}_k$ \eqref{eq_def_quantum_hilb.y} defined by
    \begin{align*}
        \mathcal{H}_k|_U := \oplus_{b \in U \cap B_k} H^0(X_b; L^k\otimes |\Lambda|^{1/2}(X_b)). 
    \end{align*}
    We denote by $P_U$ the orthogonal projection from $\mathcal{H}_k$ onto this subspace. 
    When $U$ consists of a point, $\{x\} = U$, we also write $\mathcal{H}_k|_x = \mathcal{H}_k|_{\{x\}}$ and $P_x = P_{\{x\}}$. 
\end{itemize}

\section{Preliminaries}\label{sec_preliminaries}
In this section, we recall basic facts on Lagrangian fiber bundles and its geometric quantizations. 
\begin{defn}
  Let $(X^{2n}, \omega)$ be a symplectic manifold of dimension $2n$. 
  A regular fiber bundle structure $\mu \colon X^{2n} \to B^n$ is called a {\it Lagrangian fiber bundle} if all the fibers are Lagrangian. 
\end{defn}
In what follows, we are interested in the case where a Lagrangian fiber bundle is proper with connected fibers. 

\begin{ex}\label{ex_std_lag}
Let us consider $X = \R^n \times T^n$ equipped with the symplectic form $\omega = {}^t\! dx \wedge d\theta$. 
Then $\mu \colon X \to \R^n, (x, \theta) \to x$ is a Lagrangian fiber bundle. 
\end{ex}
By Arnold-Liouville theorem \cite{Arnold1989}, any proper Lagrangian fiber bundle with connected fibers are locally isomorphic to the one in Example \ref{ex_std_lag}. Indeed, we have the following.  

\begin{fact}[\cite{Arnold1989}]\label{fact_AL}
Let $(X^{2n}, \omega)$ be a symplectic manifold and $\mu \colon X \to B$ be a proper Lagrangian fiber bundle structure with connected fibers. 
For any point $b \in B$, there exists an open neighborhood $U \subset B$ of $b$ and a symplectomorphism $(X_U, \omega) \simeq  (U' \times T^n, {}^t\! dx \wedge d\theta)$, where $U' \subset \R^n$, which makes the following diagram commutative. 
\begin{align}\label{diag_AL}
\xymatrix@=0pt{
    (X_U, \omega) &\simeq & (U' \times T^n, {}^t\! dx \wedge d\theta) \\
     \downarrow &  & \downarrow \\
     U & \simeq & U'
     }
\end{align}
\end{fact}
On a Lagrangian fiber bundle, we call a local coordinate $(x, \theta) \in \R^n \times T^n$ obtained by a local isomorphism \eqref{diag_AL} an {\it action-angle coordinate}. 
As is easy to see, any two action-angle coordinate are related as follows. 
\begin{lem}
Any symplectomorphism $\R^n \times T^n \to \R^n \times T^n$ which is compatible with the fiber bundle structure $ \mu \colon (\R^n \times T^n, {}^t\! dx \wedge d\theta) \to \R^n$ is of the form
\begin{align}\label{eq_transformation}
    (x, \theta) \mapsto ({}^t\!A^{-1}x + c, A\theta + \alpha(x)), 
\end{align}
where $A \in GL_n(\Z)$, $c \in \R^n$ and $\alpha \in C^\infty(\R^n; T^n)$ such that $A^{-1}\frac{\del \alpha}{\del x}$ is symmetric. 
\end{lem}
Thus, for a Lagrangian fiber bundle as above, the transformation between two action-angle coordinates is given by the formula of the form \eqref{eq_transformation}. 
From this, we see that the base $B$ of a Lagrangian fiber bundle admits a canonical integral affine manifold structure, and the fibers admit a canonical affine torus structure. 

Next we consider prequantum line bundles. 
Recall that a {\it prequantum line bundle} on a symplectic manifold $(X, \omega)$ is a hermitian line bundle with unitary connection $(L, \nabla)$ whose curvature satisfies $\nabla^2 = -\sqrt{-1}\omega$. 
\begin{ex}
For the case $(X, \omega) = (\R^n \times T^n, {}^t\! dx \wedge d\theta)$, we can set $(L, \nabla) = (\underline{\C}, d -\sqrt{-1}{}^t\! x  d\theta))$. 
\end{ex}
On a Lagrangian fiber bundle, a prequantum line bundle also admits a nice local description, as follows. 
\begin{lem}\label{lem_trivial_preq}
In the settings of Fact \ref{fact_AL}, we also assume that $(X, \omega)$ is equipped with a prequantum line bundle $(L, \nabla)$. 
For any point $b \in B$, there exists a contractible open neighborhood $U \subset \R^n$ of $b$ a fiber-preserving symplectomorphism $X_U\simeq U' \times T^n$ as in Fact \ref{fact_AL}, and an isomorphism $(L|_{X_U}, \nabla|_{X_U}) \simeq (\underline{\C}, d -\sqrt{-1}{}^t\! x  d\theta))$ which covers the symplectomorphism. 
\end{lem}
\begin{proof}
First let us choose an action-angle coordinate chart $X_U \simeq V \times T^n$ with $U$ contractible, and denote the coordinate by $(x', \theta')$. 
Since $U$ is contractible and the fibers are Lagrangian, $\omega|_{X_U}$ is exact. 
So $L|_{X_U}$ is trivial as a hermitian line bundle, and taking a trivialization we have $(L|_{X_U}, \nabla|_{X_U}) \simeq (V \times T^n \times \C, d - \sqrt{-1}\beta)$, where $\beta \in \Omega^1(V \times T^n)$. 
Since we have $\nabla^2 = -\sqrt{-1}\omega$, we see that $\beta - {}^t\! x  d\theta$ is closed and defines a class $(\tau_i)_{i = 1}^n$ in $H_{dR}^1(V \times T^n; \R) \simeq \R^n$, and we can write
\begin{align*}
    \beta -{}^t\! x  d\theta = \sum_{i = 1}^n \tau_i d\theta_i + df
\end{align*}
for some $f \in C^\infty(V \times T^n)$. 
Setting $U := V + (\tau_i)_i$, we define the bundle isomorphism, 
\begin{align*}
    V \times T^n \times \C \to U \times T^n \times \C, 
    \quad
    (x, \theta, v) \mapsto (x+  (\tau_i)_i, \theta, e^{-\sqrt{-1}f(x, \theta)}v ). 
\end{align*}
This gives the desired isomorphism. 
\end{proof}

Given a Lagrangian fiber bundle with a prequantum line bundle, the restriction of the prequantum line bundle to each fibers is a flat line bundle. 
Also remark that, the flat connection on the fiberwise tangent bundle $\ker d\mu$ induces the canonical flat connection on the vertical half-density bundle $|\Lambda|^{1/2} (\ker d\mu)^*$. 
\begin{defn}
Assume we are given a pre-quantized symplectic manifold $(X, \omega, L, \nabla)$ equipped with a proper Lagrangian fiber bundle structure $\mu \colon X \to B$ with connected fibers. 
\begin{enumerate}
    \item A point $b \in B$ is called a $k$-{\it Bohr-Sommerfeld point} if the space of fiberwise parallel sections of $(L^k, \nabla)$ is nontrivial.
    \item For each $k$, let $B_k \subset B$ denote the set of $k$-Bohr-Sommerfeld points.
    We define the quantum Hilbert space of level $k$ associated to the real polarization $\ker d\mu \otimes \C \subset TX \otimes \C$ by
    \begin{align}
        \mathcal{H}_k = \oplus_{b \in B_k} H^0(X_b; L^k\otimes |\Lambda|^{1/2}X_b),
    \end{align}
    where $|\Lambda|^{1/2}X_b = |\Lambda|^{1/2} (\ker d\mu)^*|_{X_b}$ is the vertical half-density bundle, equipped with the canonical flat connection. 
\end{enumerate}
\end{defn}

\begin{ex}
In the example $(X, \omega, L, \nabla) = (\R^n \times T^n, {}^t\! dx \wedge d\theta, \underline{\C}, d - \sqrt{-1}{}^t\! x  d\theta)$, the set of $k$-Bohr-Sommerfeld point is given by $B_k = \frac{\Z^n}{k} \subset \R^n$. 
For each $b \in B_k$ we have
\begin{align}\label{eq_basis_model.y}
H^0(X_b; L^k\otimes |\Lambda|^{1/2} X_b) = \C \cdot \{e^{ \sqrt{-1} k\langle b, \theta \rangle}\sqrt{d'\theta}\}, 
\end{align}
where $d'\theta := (2\pi)^{-n/2}d\theta$ be the normalized measure on $T^n$. 
The quantum Hilbert space $\mathcal{H}_k$ is the direct sum, over $B_k$, of the above one-dimensional Hilbert spaces.  
\end{ex}

\section{The construction}\label{sec_construction.y}
\subsection{The model case --- on $\R^n \times T^n$. }\label{subsec_model.y}
In this subsection, as a model case, as well as a building block of the deformation quantization solving the Problem \ref{prob_intro}, we consider the following settings. 
Let $X = \R^n \times T^n$ equipped with the standard symplectic structure ${}^t\! dx\wedge d\theta$ and the Lagrangian fibration $\mu \colon \R^n \times T^n \to \R^n$, $(x, \theta) \mapsto x$. 
Equip $X$ with the canonical prequantizing line bundle $(L = \underline{\C}, \nabla = d -\sqrt{-1} {}^t\!xd\theta)$. 

We denote by $\{\psi_b^k\}_{b \in B_k}$ the orthonormal basis of $\mathcal{H}_k$ given in \eqref{eq_basis_model.y}, namely
\begin{align*}
    \psi_b^k := e^{ \sqrt{-1} k\langle b, \theta \rangle}\sqrt{d'\theta} \in \mathcal{H}_k. 
\end{align*}

Now we construct a adjoint-preserving linear map
\begin{align*}
    \phi^k \colon C_c^\infty(X) \to \mathbb{B}(\mathcal{H}_k). 
\end{align*}
Assume we are given a function $f \in C_c^\infty(X)$. 
Using the canonical basis of $\mathcal{H}_k$ given in \eqref{eq_basis_model.y}, the operator $\phi^k(f)$ is identified by a $B_k \times B_k$-matrix $\{K_f(b, c)\}_{b, c \in B_k}$. 
Matrix elements $K_f(b, c)$ for $b, c \in B_k$ is given as follows. 
\begin{align}\label{eq_mat_elem_model.y}
    K_f(b, c) := \int_{T^n} e^{- \sqrt{-1} k \langle b - c, \theta \rangle} f((b+c)/2, \theta) d'\theta. 
\end{align}
In other words, $K_f(b, c)$ is given by the $k(b-c)$-th coefficient in the Fourier expansion of $f(c, \theta)$. 
This is regarded as a discrete analogue of the formula \eqref{eq_fourier_trans.y}. 
It is easy to see this formula takes a real-valued function to a self-adjoint operator, thus the linear map $\phi^k$ is adjoint-preserving. 
Note that we are using the value of a function at the middle point $(b + c)/2$, in order to make this map adjoint-preserving.  

The formula \eqref{eq_mat_elem_model.y} gives, a priori, a densely defined possibly unbounded operator $\phi^k(f)$ on $\mathcal{H}_k$. 
We are going to prove that that this operator is bounded in Lemma \ref{lem_bdd_DQ.y}. 

As a preparation, we give an easy estimate of norms of bounded operators. 
We are going to use this lemma throughout this paper, in order to estimate norms of operators whose entries are concentrated near the diagonal. 

\begin{lem}\label{lem_est_norm.y}
Fix a positive integer $n$. 
Let $\mathcal{H}$ be a separable Hilbert space, and assume that we are given an complete orthonormal basis $\{\psi_x\}_{x \in \Z^n}$ for $\mathcal{H}$ labelled by $\Z^n$. 
Let $\mathcal{A}$ be a possibly unbounded densely defined linear operator on $\mathcal{H}$, defined in terms of matrix coefficients with respect to the basis $\{\psi_x\}_{x\in \Z^n}$, denoted by $K(x, y)$ for $(x, y) \in \Z^n$. 
Then we have
\begin{align*}
    \|\mathcal{A}\| \le \sum_{m \in \Z^n} \left( \sup_{x \in \Z^n} \left|K(x + m, x)\right|\right). 
\end{align*}
\end{lem}

\begin{proof}
For $x \in \Z^n$, let $P_x \in \mathbb{B}(\mathcal{H})$ be the orthogonal projection onto the one-dimensional subspace $\C \cdot \psi_x \subset \mathcal{H}$. 
We decompose
\begin{align*}
    \mathcal{A} =\sum_{m \in \Z^n} \left( \sum_{x \in \Z^n} P_{x + m} \mathcal{A} P_x
    \right). 
\end{align*}
For each $m \in \Z^n$, we define the ''shift by $m$" operator $S_m \in \mathbb{U}(\mathcal{H})$ by the formula
\begin{align*}
    S_m(\psi_x) := \psi_{x + m}, 
\end{align*}
for each $x \in \Z^n$. 
For each $m$, we have
\begin{align*}
    \sum_{x \in \Z^n} P_{x + m} \mathcal{A} P_x = S_m \cdot \mathrm{diag}(\{K(x + m, x)\}_{x \in \Z^n}), 
\end{align*}
where $\mathrm{diag}(\{K(x + m, x)\}_{x \in \Z^n})$ means the diagonal operator with respect to the orthonormal basis $\{\psi_x\}_x$, whose $x$-th entry is given by $K(x + m, x)$. 
So the norm is given by
\begin{align*}
 \left\|\sum_{x \in \Z^n} P_{x + m} \mathcal{A} P_x \right\| &= 
 \left\|\mathrm{diag}(\{K(x + m, x)\}_{x \in \Z^n}) \right\| \\
 &= \sup_{x \in \Z^n} \left| K(x + m, x)\right|. 
\end{align*}
So we get
\begin{align*}
    \|\mathcal{A}\| &\le \sum_{m \in \Z^n} \left\| \sum_{x \in \Z^n} P_{x + m} \mathcal{A} P_x
    \right\| \\
    &= \sum_{m \in \Z^n} \left( \sup_{x \in \Z^n} \left| K(x + m, x)\right| \right). 
\end{align*}
\end{proof}

\begin{lem}\label{lem_bdd_DQ.y}
For $f \in C^\infty_c(X)$, the linear operator $\phi^k(f)$ defined in terms of the matrix coefficient in \eqref{eq_mat_elem_model.y} is a bounded operator on $\mathcal{H}_k$. 
\end{lem}
\begin{proof}
Let $f(x, \theta) = \sum_{m \in \Z^n} f_m(x) e^{ \sqrt{-1} \langle m, \theta \rangle}$ be the fiberwise Fourier expansion of $f$. 
Since $f$ is smooth and compactly supported, there exists a constant $C$ such that we have
\begin{align*}
    \|f_m\|_{C^0} \le \frac{C}{(1 + |m|)^{n+1}}
\end{align*}
for all $m \in \Z^n$. 
We apply Lemma \ref{lem_est_norm.y} to the operator $\phi^k(f)$ (the index set $B_k = \frac{\Z^n}{k}$ of the orthonormal basis for $\mathcal{H}_k$ is rescaled to $\Z^n$ in the obvious way). 
By \eqref{eq_mat_elem_model.y}, we have
\begin{align*}
    \left| K_f\left(x + \frac{m}{k}, x\right) \right| = \left|f_{m}\left(x + \frac{m}{2k}\right)\right|
    \le \frac{C}{\left(1 + |m|\right)^{n+1}}, 
\end{align*}
for all $x \in B_k$. 
So we get
\begin{align*}
    \|\phi^k(f)\| \le
     C \sum_{m \in \Z^n} \frac{1}{(1 + |m|)^{n+1}}
    < +\infty. 
\end{align*}
\end{proof}

Now we consider a more coordinate-free way to express \eqref{eq_mat_elem_model.y}. 
First we note the following. 
Using the product structure $X = \R^n \times T^n$ and the trivialization of $L$, for any points $b, c\in \R^n$ we get the explicit affine isomorphism
\begin{align}\label{eq_isom_base.y}
    X_b \simeq X_{(b + c)/ 2} \simeq X_c \simeq T^n, 
\end{align}
and the explicit ismomorphism of line bundles 
\begin{align}\label{eq_isom_line.y}
    L|_{X_b} \simeq L|_{X_{(b+c)/2}} \simeq L|_{X_c} \simeq T^n \times \C
\end{align}
that covers \eqref{eq_isom_base.y}. 
Using the isomorphisms \eqref{eq_isom_base.y} and \eqref{eq_isom_line.y}, we can regard a function $F_b \in C^\infty(X_b)$ or a section $\xi_b \in C^\infty(X_b; L|_{X_b})$ as an element in $C^\infty(T^n)$. 
So we can multiply a section $\xi_c \in C^\infty(X_c; L|_{X_c})$ by a function on a different fiber, $F_b \in C^\infty(X_b) $ to get an elemeent $F_b \cdot \xi_c \in C^\infty(T^n) \simeq C^\infty(X_c; L|_{X_c})$. 
Since the vertical half-density bundle $|\Lambda|^{1/2}(\ker d\mu)^*$ is equipped with the canonical flat connection, we also get a well-defined pairing of sections of $L\otimes |\Lambda|^{1/2}(\ker d\mu)^*$ between different fibers, denoted by $\langle \cdot, \cdot \rangle_{T^n}$. 

After these preparations, we proceed to give a more coordinate-free way to express \eqref{eq_mat_elem_model.y}.
The following formula follows directly from the definition. 
\begin{lem}
We have
\begin{align*}
    K_f(b, c) = \langle \psi_b^k, f|_{X_{(b + c)/2}} \cdot \psi_c^k \rangle_{T^n}, 
\end{align*}
where the left hand side is defined in \eqref{eq_mat_elem_model.y}. 
\end{lem}

\begin{defn}\label{def_rep_model.y}
For $k \in \N$, we define the adjoint-preserving linear map
\begin{align*}
    \phi^k \colon C_c^\infty(X) \to \mathbb{B}(\mathcal{H}_k)
\end{align*}
by the formula
\begin{align}\label{eq_total_model.y}
    \phi^k(f)(\psi^k_c) = \sum_{b \in B_k}\langle \psi_b^k, f|_{X_{(b + c)/2}} \cdot \psi_c^k \rangle_{T^n} \cdot \psi^k_b. 
\end{align}
\end{defn}

\begin{ex}
Assume $f \in C_c^\infty(X)$ is a pullback of a function $f_0 \in C_c^\infty(\R^n)$ on the base $\R^n$, i.e., $f$ does not depend on $\theta$. 
Then $\phi^k(f)$ is just the diagonal multiplication operator by the value of $f_0$ at each point on $B_k$, 
\begin{align*}
    K_f(b, c) = \begin{cases}
    f_0(c) & \mbox{ if } b = c, \\
    0 & \mbox{ otherwise. }
    \end{cases}
\end{align*}
\end{ex}
\begin{ex}\label{ex_concentration}
Assume $f$ can be expressed as $f(x, \theta) = f_m(x) e^{\sqrt{-1}\langle m, \theta \rangle}$ for some $m \in \Z^n$ and a function $f_m \in C^\infty_c(\R^n)$.
Then we have
\begin{align*}
    K_f(b, c) = 
    \begin{cases}
    f_m\left(c + m / (2k)\right) & \mbox{ if } b = c + m/k, \\
    0 & \mbox{ otherwise. }
    \end{cases}
\end{align*}
We see that the function $e^{\sqrt{-1}\langle m, \theta \rangle}$ plays the role of ''$m/k$-shift", and if we let $k \to \infty$, the matrix elements of this operator concentrate to the diagonal. 

More generally, if $f$ can be expressed as $f = \sum_{|m| \le M} f_m(x) e^{\sqrt{-1}\langle m, \theta \rangle}$ for some $M < \infty$, we have $K_f(b, c) \neq 0$ only when $|b - c| \le M/ k$. 
So also in this case the matrix elements concentrate to the diagonal as $k \to \infty$. 
\end{ex}
In fact, the ''concentration to the diagonal" of the matrix elements of the operator $\phi^k(f)$ as $k\to \infty$ seen in the above examples holds in general, because the Fourier coefficients of smooth function on $T^n$ is rapidly decreasing. 
Basically, this is why we can extend this construction to general Lagrangian fiber bundles in the next subsection. 

The goal of the rest of this subsection is to prove that the maps $\{\phi^k\}_{k \in \N}$ is a strict deformation quantization of $(X, \omega)$. 
Actually, in this model case, it is easy to explicitely compute the asymptotic behavior of $\phi^k(f)\phi^k(g)$ as $k \to \infty$ for $f, g \in C_c^\infty(X)$. 
We show that this recovers the standard Moyal-Weyl star product on $\left( C^\infty(\R^n \times T^n)[[ \hbar ]], {}^t\!dx \wedge d\theta \right)$. 

First we recall the definition of Moyal-Weyl star product on $(\R^{2n}, \omega)$ with a translation invariant symplectic form $\omega = \frac{1}{2} \omega_{ij}dx^i \wedge dx^j$. 
For functions $f, g \in C^\infty(\R^{2n})$, the Moyal-Weyl star product is defined by
\begin{align*}
    (f * g) (x) := \exp \left(\frac{\hbar}{2} \omega^{ij}\del_{y^i}\del_{z^j}\right)f(y)g(z)|_{y = z = x}. 
\end{align*}
We denote the coefficient of $\hbar^j$ of the star product by $\mathcal{C}_j(f, g) \in C^\infty(\R^{2n})$, so that
\begin{align*}
    f*g = \sum_{j = 0}^\infty \mathcal{C}_j(f, g) \hbar^j. 
\end{align*}
Consider case where $\R^{2n} = \R^n \times \R^n $ with the coordinate $(x, y)$ and the symplectic form is the standard one $\omega = {}^t\! dx \wedge dy$. 
This induces a star product on the quotient space, $(\R^n \times T^n, {}^t\!dx \wedge d\theta)$, denoted by $*_{\mathrm{std}}$. 
This is given by the formula, 
\begin{align} \label{eq_std_MW.y}
    (f *_{\mathrm{std}} g) (x, \theta) := \left.\exp \left(\frac{\hbar}{2}\sum_{i}\left( \del_{x'_i}\del_{\theta''_i} -\del_{x''_i}\del_{\theta'_i}\right)  \right)f(x', \theta')g(x'', \theta'')\right|_{(x', \theta') = (x'', \theta'') =(x, \theta) }. 
\end{align}
We still call this star product as the standard Moyal-Weyl star product. 
We denote by $\mathcal{C}^{\mathrm{std}}_j(f, g)\in C^\infty(\R^n \times T^n)$ the corresponding coefficient with respect to the star product $*_{\mathrm{std}}$ for $f , g \in C^\infty(\R^n \times T^n)$. 

\begin{prop}\label{prop_model_star.y}
The linear map $\phi^k\colon C_c^\infty(X) \to \mathbb{B}(\mathcal{H}_k)$ defined in Definition \ref{def_rep_model.y} satisfies, for all $f , g \in C_c^\infty(X)$ and $l \in \N$, 
\begin{align*}
        \left\| \phi^k(f) \phi^k(g) -
        \sum_{j = 0}^l \left( \frac{-\sqrt{-1}}{k}\right)^j \phi^k\left(\mathcal{C}_j^{\mathrm{std}}(f, g)\right)
        \right\| =O\left(\frac{1}{k^{l+1}}\right)
\end{align*}
as $k \to \infty$. 
\end{prop}

\begin{proof}
Fix $f, g$ and $l$. 
We denote the Fourier expansion of $f$, $g$ by $f_m$, $g_m$ as usual. 
The matrix coefficient of $\phi^k(f) \phi^k(g)$, denoted by $K_{\phi^k(f) \phi^k(g)}(\cdot, \cdot)$, is given by
\begin{align}\label{eq_comp_model_star.y}
    K_{\phi^k\left(f\right)\phi^k\left(g\right)}\left(x + \frac{p}{k}, x\right)
    &=\sum_{m \in \Z^n} K_f\left(x + \frac{p}{k},x + \frac{m}{k} \right) K_g\left(x + \frac{m}{k}, x\right) \notag \\
    & \quad = \sum_{m \in \Z^n} f_{p-m}\left(x + \frac{p+m}{2k}\right)g_{m}\left(x+ \frac{m}{2k}\right).
\end{align}
The proof is given by performing the Taylor expansion of $f_{p-m}$ and $g_m$ around the point $x + \frac{p}{2k}$ in the above formula. 

For simplicity, we only give the proof in the case $n = 1$. 
The proof is essentially the same for general $n$. 
Define the operator $B_{err}^{k, l}$ on $\mathcal{H}_k$ by
\begin{align*}
    B_{err}^{k, l} :=  \phi^k(f) \phi^k(g) -
        \sum_{j = 0}^l \left( \frac{-\sqrt{-1}}{k}\right)^j \phi^k\left(\mathcal{C}_j^{\mathrm{std}}(f, g)\right). 
\end{align*}

The standard Moyal-Weyl star product is explicitly given by
\begin{align*}
    \mathcal{C}_j^{\mathrm{std}}(f, g) = 
    \frac{1}{2^j}\sum_{i = 0}^j \frac{(-1)^{i-j}}{j! (i-j)!}\left(\del_{x}^i \del_\theta^{j-i} f \right) \cdot \left(\del_{x}^{j-i} \del_\theta^{i} g \right)
\end{align*}
So the $p$-th Fourier coefficient of this function is given by
\begin{align}\label{eq_fourier_star.y}
    \left(\mathcal{C}_j^{\mathrm{std}}(f, g)\right)_p
    = \left(\frac{\sqrt{-1}}{2}\right)^j\sum_{i = 0}^j \frac{(-1)^{i-j}}{j! (i-j)!}
    \sum_{m \in \Z} (p-m)^{j - i} f_{p-m}^{(i)} \cdot m^i g_m^{(j-i)}. 
\end{align}
On the other hand, the Taylor expansion of $f_{p-m}$ and $g_m$ gives, formally, 
\begin{align*}
    f_{p-m}\left(x + \frac{p+m}{2k}\right) &= \sum_{i = 0}^\infty \frac{1}{i!} \left( \frac{m}{2k}\right)^i f^{\left(i\right)}_{p-m}\left(x + \frac{p}{2k}\right) &\mbox{(formally)}\\
    g_m\left(x + \frac{m}{2k}\right)&= \sum_{i = 0}^\infty \frac{1}{i!}\left( \frac{m-p}{2k}\right)^i g^{\left(i\right)}_{m}\left(x + \frac{p}{2k}\right) &\mbox{(formally)}. 
\end{align*}
Thus the formula \eqref{eq_comp_model_star.y} admits an expansion, at least formally, 
\begin{align}\label{eq_formal_comp_star.y}
    K_{\phi^k(f)\phi^k(g)}\left(x + \frac{p}{k}, x\right)
    = \sum_{j = 0}^\infty\left(\frac{1}{2k} \right)^j
    \sum_{i = 0}^j \frac{(-1)^{i-j}}{j! (i-j)!}
    \sum_{m \in \Z} (p-m)^{j - i} f_{p-m}^{(i)} \cdot m^i g_m^{(j-i)}. 
\end{align}
By \eqref{eq_fourier_star.y} and \eqref{eq_formal_comp_star.y}, we see that, in the above formal expansion, the matrix elements $K_{B_{err}^{k, l}}(x + \frac{p}{k}, x)$ of $B_{err}^{k,l}$ is $O(k^{-(l+1)})$ with respect to $k$. 

Now we give estimates for the error terms and prove the statement. 
For simplicity we only prove in the case $n = 1$ and $l = 1$. 
The proof is essentially the same for the general case.
Put $\mathcal{A}_{err}^k := k^2 B_{err}^{k, 1}$. 
Namely we have
\begin{align*}
    \mathcal{A}_{err}^k := k^2\left( \phi^k(f) \phi^k(g) -\phi^k(fg)- \frac{-\sqrt{-1}}{2 k}\phi^k(\{f, g\}) \right). 
\end{align*}
We need to show $\sup_k \|\mathcal{A}_{err}^k\| < + \infty$. 
We denote by $K_{err}^k(x, y)$ the matrix element of $\mathcal{A}_{err}^k$ for $(x, y) \in B_k \times B_k$. 

Since $f$ and $g$ are smooth and compactly supported, for each $N \in \N$ there exists a constant $C_N$ such that, for all $m \in \Z$ we have
\begin{align}\label{est_coeff_star.y}
    \|f_m\|, \|g_m\|, \|f'_m\|, \|g'_m\|, \|f''_m\|, \|g''_m\| < C_N(|m|+1)^{-N}, 
\end{align}
where $\|\cdot\|$ denotes the $C^0$-norm.

By \eqref{est_coeff_star.y}, we have
\begin{align*}
    \left| f_{p-m}\left(x + \frac{p+m}{2k}\right) - \left(f_{p-m}\left(x + \frac{p}{2k}\right) + \frac{m}{2k}f'_{p-m}\left(x +\frac{p}{2k}\right)
    \right)\right|
    \le \frac{|m|^2}{8k^2} \cdot \frac{C_4}{\left(1 + |p-m|\right)^4} \\
    \left| g_{m}\left(x + \frac{m}{2k}\right) - \left(g_{m}\left(x + \frac{p}{2k}\right) - \frac{p-m}{2k}g'_{m}\left(x +\frac{p}{2k}\right)
    \right)\right|
    \le \frac{|p-m|^2}{8k^2}\cdot \frac{C_6}{\left(1 + |m|\right)^6}
\end{align*}
We also have the estimate
\begin{align*}
    \|f_{p-m}\| + \frac{|m|}{2k} \|f'_{p-m}\| \le 
    \frac{(1 + |m|)C_4}{(1 + |p-m|)^4}, \mbox{ and }
    \|g_m\| + \frac{|p-m|}{2k}\|g_m'\| \le \frac{(1 + |p-m|)C_6}{(1 + |m|)^6}
\end{align*}
Combining the above estimates and \eqref{eq_comp_model_star.y} and \eqref{eq_fourier_star.y} for $j = 0, 1$, we have
\begin{align*}
    &\left| K_{err}^k\left(x + \frac{p}{k}, x\right) \right| \\
    &\le k^2 \sum_{m \in \Z} \left( \frac{|m(p-m)|}{4k^2} \|f'_{p-m}\| \cdot \|g'_m\| + 
    \frac{C_4C_6}{8k^2(1 + |m|)^4(1 + |p-m|)^3}
     \right.\\
    &\qquad \qquad \left.+\frac{C_4C_6}{8k^2(1 + |m|)^5(1 + |p-m|)^2}
    + \frac{C_4C_6}{64k^4(1 + |m|)^4(1 + |p-m|)^2}
    \right) \\
    &\le \sum_{m \in \Z}C_4C_6\left( \frac{1}{4} + \frac{1}{8} + \frac{1}{8} + \frac{1}{64}
    \right)\frac{1}{(1 + |m|)^4(1 + |p-m|)^2} \\
    &\le \frac{C'}{(1 + |p|)^2}, 
\end{align*}
where we put
\begin{align*}
    C' := C_4C_6\left( \frac{1}{4} + \frac{1}{8} + \frac{1}{8} + \frac{1}{64}
    \right)
 \sum_{m \in \Z} \frac{1}{(1 + |m|)^2} < +\infty
\end{align*}
and used the inequality $(1 + |m|)(1 + |p-m|) \ge 1 + |p|$. 
So we have, by Lemma \ref{lem_est_norm.y}, 
\begin{align*}
    \|\mathcal{A}_{err}^k\| 
    \le \sum_{p \in \Z} \left(\sup_{x \in B_k} \left| K_{err}^k\left(x + \frac{p}{k}, x\right)\right| \right)
    \le \sum_{p \in \Z}\frac{C'}{(1 + |p|)^2}. 
\end{align*}
This gives a bound for $\|\mathcal{A}_{err}^k\|$ which does not depend on $k$, so we get the result. 

\end{proof}

\begin{prop}\label{prop_model_norm.y}
The linear map $\phi^k\colon C_c^\infty(X) \to \mathbb{B}(\mathcal{H}_k)$ defined in Definition \ref{def_rep_model.y} satisfies, for all $f\in C_c^\infty(X)$,  
\begin{align*}
\|\phi^k(f)\| \to \|f\|_{C^0}. 
\end{align*}
as $k \to \infty$. 
\end{prop}

\begin{proof}
First we observe the following. 
\begin{lem}\label{lem_const_model.y}
Let $F = \sum_{m \in \Z^n} F_m e^{\sqrt{-1}   \langle m, \theta \rangle} \in C^\infty(T^n)$, and consider the operator $\Phi^k(F) \in \mathbb{B}(\mathcal{H}_k)$ such that the corresponding matrix coefficient, denoted by $K_F(x, y)$ for $(x, y) \in B_k$, is given by
\begin{align*}
    K_F\left(x + \frac{m}{k}, x\right) := F_m, 
\end{align*}
for all $x \in B_k$ and $m \in \Z^n$. 
Then we have
\begin{align*}
    \|\Phi^k(F)\| = \|F\|_{C^0}. 
\end{align*}
\end{lem}
Notice that if we regard $F$ as a function on $X = \R^n \times T^n$ by $\tilde{F}(x, \theta) = F(\theta)$, then the operator $\Phi^k(F)$ should be regarded as ''$\phi^k(\tilde{F})$", even though $\phi^k(\tilde{F})$ is not defined because $\tilde{F}$ is not compactly supported.  
\begin{proof}
We have a unitary isomorphism
\begin{align*}
    \mathcal{H}_k &\simeq L^2(T^n) \\
    \psi^k_{m/k} &\mapsto e^{ \sqrt{-1} \langle m, \theta \rangle}. 
\end{align*}
Under this isomorphism, the operator $\Psi^{k}(F)$ transforms to the multiplication operator by $F$ on $L^2(T^n)$, and this operator norm is $\|F\|_{C^0}$. 
\end{proof}

Proposition \ref{prop_model_norm.y} is, very roughly, understood as follows. 
As $k \to \infty$, the operator $\phi^k(f)$ reflects the behavior of $f$ only locally in $\R^n$-direction (see Example \ref{ex_concentration}), and on a sufficiently small neighbourhood of a point $x \in \R^n$, the function $f$ is close to a function which is invariant in $\R^n$-direction, and Lemma \ref{lem_const_model.y} applies asymptotically. 
Since the proof is long (although very elementary) and not essential for the rest of the paper, we give a detailed proof in the Appendix.  
\end{proof}

Combining Proposition \ref{prop_model_star.y} and Proposition \ref{prop_model_norm.y}, we get the following. 
\begin{thm}\label{thm_model}
The family of adjoint-preserving linear maps $\{\phi^k\}_{k \in \N}$, $\phi^k \colon C_c^\infty(\R^n \times T^n) \to \mathbb{B}(\mathcal{H}_k)$, defined in Definition \ref{def_rep_model.y}, is a strict deformation quantization for $(\R^n \times T^n, {}^t\!dx\wedge d\theta)$. 
\end{thm}

\subsection{The general case}\label{subsec_construction_general.y}
In this subsection we generalize the construction of the strict deformation quantization for the general case. 
Let $(X^{2n}, \omega, L, \nabla, h)$ be a prequantized symplectic manifold equipped with a proper Lagrangian fiber bundle $\mu \colon X \to B$ with connected fibers. 
The following two additional datum are needed for the construction of our strict deformation quantization. 
\begin{enumerate}
    \item[(H)] A smooth horizontal distribution $H \subset TX$ for the fibration $\mu$,  
\begin{align}\label{eq_splitting.y}
    TX = H \oplus \ker d\mu,  
\end{align}
which is, choosing any local action-angle coordinate $(x, \theta) \in U' \times \T^n$, invariant in the $T^n$-direction. 
This condition does not depend on the choice of the coordinate, and it is clear that such a splitting always exists. 

\item[(U)] An open covering $\mathcal{U}$ of $B$ with the following properties. 
\begin{enumerate}
    \item[(U1)] $\mathcal{U}$ is locally finite. 
    Moreover, each element $U \in \mathcal{U}$ admits an open affine embedding into $\R^n$ and its image is convex in $\R^n$. 
    \item[(U2)] For any two elements $U, V \in \mathcal{U}$, the intersection $U \cap V$ also admits an open affine embedding into $\R^n$ with convex image (in particular it is connected). 
    \item[(U3)]For each $U \in \mathcal{U}$, $\overline{U}$ is a compact subset of $B$. 
\end{enumerate}
\end{enumerate}

We are going to construct a deformation quantization from these datum. 
As we will see, the choice of the horizontal distribution $H$ is essential for our construction, but the choice of $\mathcal{U}$ is only technical, and the different choice of $\mathcal{U}$ yields essentially the same deformation quantization (Proposition \ref{prop_open_cov.y}). 
We also remark that we can drop the condition (U3) and just require that each $U$ admits an affine open embedding into $\R^n$ (see Remark \ref{rem_open_cov.y}). 
We require this condition just in order to simplify the estimates. 

Given a path $\gamma$ in $B$ from $b\in B$ to $c \in B$, by the splitting \eqref{eq_splitting.y}, we get the parallel transform 
\begin{align}\label{eq_para.y}
    T_\gamma \colon X_b \to X_c, 
\end{align}
which preserves the affine structure. 
Also the connection $\nabla$ on $L$ and the canonical flat connection on $|\Lambda|^{1/2}(\ker d\mu)^*$ gives the parallel transform
\begin{align*}
    T_\gamma \colon L^k|_{X_b} \otimes |\Lambda|^{1/2}X_b \to L^k|_{X_c} \otimes |\Lambda|^{1/2}X_c 
\end{align*}
which covers \eqref{eq_para.y}. 
We use the same notation for the parallel transform. 
This allows us to define a pairing between sections $\xi_b^k \in C^\infty(X_b; L^k \otimes |\Lambda|^{1/2}X_b)$ and $\xi_c^k \in C^\infty(X_c; L^k \otimes |\Lambda|^{1/2}X_c)$, denoted by $\langle \xi_b^k, \xi_c^k \rangle_\gamma$. 

We say that two points $b, c \in B$ are {\it close} if there exists an element $U \in \mathcal{U}$ such that $b, c \in U$. 
For such $b, c \in B$, we can take the unique affine linear path $\gamma$ from $b$ to $c$ in $U$ and define, for sections $\xi_b^k \in C^\infty(X_b; L^k \otimes |\Lambda|^{1/2}X_b)$ and $\xi_c^k \in C^\infty(X_c; L^k \otimes |\Lambda|^{1/2}X_c)$,
\begin{align*}
    \langle \xi_b^k, \xi_c^k \rangle_\mathcal{U} := \langle \xi_b^k, \xi_c^k \rangle_\gamma. 
\end{align*}
This is well-defined by our assumptions on $\mathcal{U}$. 

\begin{defn}\label{def_rep_gen.y}
Given a splitting as \eqref{eq_splitting.y} and a covering $\mathcal{U}$ of $B$ as above, we define a adjoint-preserving linear map $\phi^k_{H, \mathcal{U}} \colon C_c^\infty(X) \to \mathbb{B}(\mathcal{H}_k)$ by the following formula. 
For $f \in C_c^\infty(X)$, we define the operator $\phi^k_{H, \mathcal{U}}(f)$ by, for $c \in B_k$ and an element $\psi^k_c \in H^0(X_c; L^k\otimes |\Lambda|^{1/2}X_c) \subset \mathcal{H}_k$,  
\begin{align*}
\phi^k_{H, \mathcal{U}}(f)(\psi^k_c) := \sum_{b \in B_k, b\mbox{ is close to }c} \langle \psi^k_b, f|_{X_{(b + c)/2}}\psi^k_c \rangle_{\mathcal{U}}  \cdot \psi^k_b, 
\end{align*}
where $\psi^k_b \in H^0(X_b; L^k\otimes |\Lambda|^{1/2}(X_b)) \subset \mathcal{H}_k$
is any element with $\|\psi^k_b\| = 1$ (this definition does not depend on this choice). 
Here, we denote by $(b + c)/2 \in B$ the middle point between $b$ and $c$ with respect to the affine structure on an open set $U \in \mathcal{U}$ which contains both $b$ and $c$, and we regard $f|_{X_{(b + c)/2}} \in C^\infty(X_{(b+c)/2})$ as a function on $X_c$ using the parallel transform \eqref{eq_para.y} along the affine linear path between $(b + c)/2$ and $c$ in $U$. 
\end{defn}

Again it is easy to see that this map is adjoint-preserving, thanks to the fact that we are using the value of function at the fiber over the middle point.  

\begin{ex}
If a function $f \in C_c^\infty(X)$ is a pullback of a function on $B$, then $\phi^k_{H, \mathcal{U}}(f)$ is the multiplication operator by its value at $B_k$. 
In other words, if we write $f = \mu^*f_0$ for $f_0 \in C_c^\infty(B)$ we have
\begin{align*}
    \phi^k_{H, \mathcal{U}}(\mu^*f_0) \psi_b^k = f_0(b) \psi_b^k
\end{align*}
for any $b \in B_k$ and $\psi_b^k \in H^0(X_b; L^k \otimes |\Lambda|^{1/2}X_b) \subset \mathcal{H}_k$. 
\end{ex}

The goal of this subsection is to prove that the family of maps defined in Definition \ref{def_rep_gen.y} is a strict deformation quantization. 
From now on until the end of this subsection, we fix $H$ and $\mathcal{U}$ satisfying the conditions in (H) and (U) above, and moreover, 
\begin{itemize}
    \item We fix an action-angle coordinate on $X_U$ and a trivialization $(L|_{X_U}, \nabla|_{X_U}) = (\underline{\C}, d - \sqrt{-1}{}^t\! x d\theta)$ for each $U \in \mathcal{U}$ (see Lem \ref{lem_trivial_preq}). 
\end{itemize}

Let us focus on an element $U \in \mathcal{U}$. 
Since $H$ is invariant in the fiber direction, using the fixed action-angle coordinate we can write
$H$ as
\begin{align}\label{eq_def_A}
    H = \mathrm{Span} \left\{ \frac{\del}{\del x_j} + A_j^i \frac{\del}{\del \theta_i}\right\}_{1 \le j \le n}, 
\end{align}
for some $A_j = (A_j^i) \in C^\infty(U; \R^n)$, $j = 1, \cdots, n$. 
We regard $A = A_jdx_j \in C^\infty(U; T^*U\otimes \R^n )$. 
Using the flat connection on $U \subset B$, we get $\nabla A|_U \in C^\infty(U; T^*U \otimes T^*U\otimes \R^n )$ given by
\begin{align*}
    \nabla A^i = \frac{\del A_j^i}{\del x_l} dx_l \otimes dx_j . 
\end{align*}
We denote by $\|\nabla A\|_{U}$ the $C^0(U)$-norm of $\nabla A$
with respect to the flat metric on $U$ induced by the Euclidean metric of the action coordinate. 
Remark that $A$ is only defined on $U$ and depends on the action-angle coordinate chosen on $U$. 

\begin{lem}\label{lem_arg.y}
Fix an element $U \in \mathcal{U}$. 
Using the fixed action-angle coordinate and trivialization on $X_U$ as above, take an orthonormal basis $\{\psi^k_x\}_{x \in \frac{\Z^n}{k} \cap U}$ of $\mathcal{H}^k|_{U}$ given by
\begin{align*}
    \psi^k_{x} = e^{ \sqrt{-1} k\langle x, \theta \rangle} \sqrt{d'\theta}
    \in H^0(X_x; L^k \otimes |\Lambda|^{1/2}(X_x)). 
\end{align*}
For a function $f \in C_c^\infty(X)$, denote the matrix coefficient of $\phi^k_{H, \mathcal{U}}(f)|_{\mathcal{H}^k_U}$ with respect to the above orthonormal basis by $\{K_f^k(x, y)\}_{x, y \in \frac{\Z^n}{k} \cap U}$. 
Denote the Fourier expansion of $f|_{X_U}$ by $f(x, \theta) = \sum_{m \in \Z^n} f_m(x) e^{\sqrt{-1} \langle m, \theta \rangle}$. 

Then we have, for any points $x, x + \frac{m}{k} \in \frac{\Z^n}{k} \cap U$ such that $K_f^k(x + \frac{m}{k}, x) \neq 0$, 
\begin{align}\label{eq_U(1).y}
    \frac{K_f^k(x + \frac{m}{k}, x)}{f_m(x + \frac{m}{2k})}  \in U(1), 
\end{align}
and the value of \eqref{eq_U(1).y} does not depend on $f$ as long as $f_m(x + \frac{m}{2k}) \neq 0$. 
Moreover, we have
\begin{align}\label{eq_arg.y}
    \left| \arg \left(\frac{K_f^k(x + \frac{m}{k}, x)}{f_m(x + \frac{m}{2k})} \right)\right| \le \frac{5}{24} \|\nabla A\|_U\frac{|m|^3}{k^2}.  
\end{align}
Here $A \in C^\infty(U; T^*U \otimes \R^n)$ is defined in \eqref{eq_def_A}. 
\end{lem}

\begin{proof}
When $m = 0$ the result is obvious. 
Fix $k$, a point $x \in \frac{\Z^n}{k} \cap U$ and $m \neq 0 \in \Z^n$ with $x + \frac{m}{k} \in \frac{\Z^n}{k} \cap U$. 
Let $\bm{u} := \frac{m}{|m|}$ be the unit vector in the $m$-direction. 
Denote by $\gamma = [x, x + \frac{m}{k}]$ the line segment from $x$ to $x + \frac{m}{k}$ in $U$, and denote by $\tilde{\gamma} \colon [x, x + \frac{m}{k}] \to U \times T^n = X_U$ the horizontal lift of $[x, x + \frac{m}{k}]$ with respect to the splitting \eqref{eq_splitting.y} which passes through the point $(x + \frac{m}{2k}, 0)$. 
Define $\alpha \colon \left[-\frac{|m|}{2k}, \frac{|m|}{2k} \right] \to T^n $ by
\begin{align*}
    \tilde{\gamma}\left(t\bm{u} + x + \frac{m}{2k}\right) = \left(t \bm{u} + x + \frac{m}{2k}, \alpha\left(t\right)\right).  
\end{align*}
We have $\alpha (0) = 0$. 
We regard $\bm{u} := m/|m| \in C^\infty(U; TU)$, and $\alpha$ satisfies
\begin{align}\label{eq_alpha_der.y}
    \alpha'(t) &= A(\bm{u})|_{t\bm{u} + x + m/(2k)},  \\
    \alpha''(t) &= \nabla_{\bm{u}} A (\bm{u})|_{t\bm{u} + x + m/(2k)}.  \notag
\end{align}
So we get, for all $t \in [-\frac{|m|}{2k}, \frac{|m|}{2k} ]$ 
\begin{align}\label{eq_alpha.y}
   | \alpha(t) - \alpha'(0)t | &\le \frac{ \|\nabla A\|_{U}}{2}t^2,  \\ 
   |\alpha'(t) - \alpha'(0)| & \le \|\nabla A\|_{U} t. \notag
\end{align}
We are going to compute the pairings between elements of $C^\infty(X_{x}; L^k \otimes |\Lambda|^{1/2}X_x)$ and $C^\infty(X_{x + \frac{m}{k}}; L^k \otimes |\Lambda|^{1/2}X_{x+ \frac{m}{k}})$ by pulling back to the middle fiber $X_{x + \frac{m}{2k}}$. 
So from now on, only in this proof we write the parallel transport along subsets of the segment $\gamma$ by
\begin{align*}
    T_{t} &\colon  X_{x +\frac{m}{2k}} \to X_{t\bm{u} + x + \frac{m}{2k}}, \\
    T_{t} &\colon  L|_{X_{x +\frac{m}{2k}}} \to L|_{X_{t\bm{u} + x + \frac{m}{2k}}}. 
\end{align*}
Then, for $t \in \left[-\frac{|m|}{2k}, \frac{|m|}{2k}\right]$ and a function $g\in C^\infty(X_{t\bm{u} + x + \frac{m}{2k}})$, we have
\begin{align}\label{eq_trans_func.y}
    (T_{t}^*g) (\theta) 
    = g(\theta + \alpha(t)). 
\end{align}
Let $E \in C^\infty(X_U; L)$ be the section which gives the trivialization of $L$ as in the statement. 
Then, since it satisfies $\nabla E = -\sqrt{-1}{}^t\!xd\theta \otimes E$, the section $\widetilde{E^k} \in C^\infty(X_{\gamma}; L^k)$ given by
\begin{align*}
    \widetilde{E^k}|_{X_{t\bm{u} + x + \frac{m}{2k}}} := \exp\left(\sqrt{-1}k\int_0^t \left\langle s\bm{u} + x + \frac{m}{2k}, \alpha'(s) \right\rangle ds\right)E^k, 
\end{align*}
satisfies 
\begin{align*}
    T_t^*\widetilde{E^k}|_{X_{t\bm{u} + x + \frac{m}{2k}}} = E^k|_{X_{x + \frac{m}{2k}}}. 
\end{align*}
We have
\begin{align}\label{eq_trans_basis.y}
    \psi^k_x &= \exp \left(  \sqrt{-1} k\langle x, \theta\rangle - \sqrt{-1}k\int_0^{-|m|/2k} \left\langle s\bm{u} + x + \frac{m}{2k}, \alpha'(s) \right\rangle ds\right) \widetilde{E^k} \sqrt{d'\theta}, \\
    \psi^k_{x + \frac{m}{k}} &= \exp \left( \sqrt{-1}  k\left\langle x + \frac{m}{k}, \theta\right\rangle - \sqrt{-1}k\int_0^{|m|/2k} \left\langle s\bm{u} + x + \frac{m}{2k}, \alpha'(s) \right\rangle ds\right) \widetilde{E^k} \sqrt{d'\theta}. \notag
\end{align}
Combining \eqref{eq_trans_func.y} and \eqref{eq_trans_basis.y}, for $f(x, \theta) = \sum_{m \in \Z^n}f_m(x)e^{\sqrt{-1}\langle m, \theta \rangle} $ we have
\begin{align}
   & K_f^k(x + \frac{m}{k}, x) \notag \\
   &= \left\langle \left( T^*_{\frac{|m|}{2k}}\psi_{x+\frac{m}{k}}^k\right), \left( f|_{X_{x + \frac{m}{2k}}} \cdot T^*_{-\frac{|m|}{2k}}\psi^k_{x}\right) \right\rangle \notag \\
   & = \exp \Bigl(  \sqrt{-1} k\Big\{-\left\langle x + \frac{m}{k}, \theta + \alpha\left(\frac{|m|}{2k}\right)\right\rangle +\int_0^{|m|/2k} \left\langle s\bm{u} + x + \frac{m}{2k}, \alpha'\left(s\right) \right\rangle ds \notag\\
   & \qquad +  \left\langle x, \theta + \alpha\left(-\frac{|m|}{2k}\right)\right\rangle - \int_0^{-|m|/2k} \left\langle s\bm{u} + x + \frac{m}{2k}, \alpha'\left(s\right) \right\rangle ds \Bigr\}+ \sqrt{-1}  \left\langle m, \theta \right\rangle\Bigr) 
   \cdot f_m\left(x + \frac{m}{2k}\right) \notag\\
   &= \exp \left(\sqrt{-1}k\left\{ -  \left\langle \frac{m}{2k}, \alpha\left(\frac{|m|}{2k}\right)+\alpha\left(-\frac{|m|}{2k}\right) \right\rangle
   + \int_{-|m|/2k}^{|m|/2k} \left\langle s\bm{u} , \alpha'\left(s\right) \right\rangle ds
   \right\}\right)\cdot f_m\left(x + \frac{m}{2k}\right). \label{eq_phase.y}
\end{align}
This implies \eqref{eq_U(1).y}, as well as the independence of the value \eqref{eq_U(1).y} on $f$. 
Moreover, by \eqref{eq_alpha.y}, we have
\begin{align*}
    \left|\left\langle \frac{m}{2k}, \alpha\left(\frac{|m|}{2k}\right)+\alpha\left(-\frac{|m|}{2k}\right) \right\rangle\right| 
    &\le \|\nabla A\|_U\left( \frac{|m|}{2k}\right)^3,  \\
    |\int_{-|m|/2k}^{|m|/2k} \langle s\bm{u} , \alpha'(s) \rangle ds|& \le\int_{-|m|/2k}^{|m|/2k} \|\nabla A\|_Us^2 ds = \frac{2\|\nabla A\|_U}{3}\left( \frac{|m|}{2k}\right)^3. 
\end{align*}
Thus we get,  
\begin{align*}
    \left| \arg \left(\frac{K_f^k(x + \frac{m}{k}, x)}{f_m(x + \frac{m}{2k})} \right)\right| \le \frac{5}{24} \|\nabla A\|_U\frac{|m|^3}{k^2}, 
\end{align*}
which proves \eqref{eq_arg.y}. 
\end{proof}
Now we state our main theorem. 

\begin{thm}\label{thm_main_v2.y}
The linear map $\phi^k_{H, \mathcal{U}}\colon C^\infty(X) \to \mathbb{B}(\mathcal{H}_k)$ given in Definition \ref{def_rep_gen.y} satisfies, for all $f , g \in C_c^\infty(X)$,  
\begin{align*}
    \left\|\phi^k_{H, \mathcal{U}}(f)\phi^k_{H, \mathcal{U}}(g) - \phi^k_{H, \mathcal{U}}(fg) - \frac{-\sqrt{-1}}{2 k}\phi^k_{H, \mathcal{U}}(\{f, g\})\right\|
    = O\left(\frac{1}{k^2}\right)
\end{align*}
    as $k \to \infty$. 
\end{thm}

\begin{proof}
In this proof, we write $\phi^k = \phi^k_{H, \mathcal{U}}$. 
First, informally we compute, for $x, x + \frac{p}{k} \in U$ for some $U\in \mathcal{U}$, using Lemma \ref{lem_arg.y}, and denoting by $K_{\phi^k(f)\phi^k(g)}(\cdot, \cdot)$ the matrix coefficient of $\phi^k(f)\phi^k(g)$, 
\begin{align*}
    &K_{fg}\left(x +\frac{p}{k}, x\right) \\
    &=\left(1 + O\left(\frac{|p|^3}{k^2}\right)\right)\sum_{m \in \Z^n}\left. \left(f_{p-m}g_{m}\right)\right|_{x + \frac{p}{2k}},  \\
    &K_{\{f, g\}}\left(x +\frac{p}{k}, x\right) \\
    &=  \sqrt{-1} \left(1 + O\left(\frac{|p|^3}{k^2}\right)\right)\sum_{m \in \Z^n}\left. \left\{ \langle m, \nabla f_{p-m}\rangle \cdot g_{m} - f_{p-m}\cdot \langle \left(p-m\right), \nabla g_{m} \rangle  \right\}\right|_{x + \frac{p}{2k}},  \\
    &K_{\phi^k(f)\phi^k(g)}\left(x + \frac{p}{k}, x\right) \\
    &= \sum_m f_{p-m}\left(x + \frac{p +m}{2k}\right)g_m\left(x + \frac{m}{2k}\right)\left(1 + O\left(\frac{|p-m|^3 + |m|^3}{k^2}\right)\right) + (\mbox{error}) \\
    &= \sum_m  \left(1 + O\left(\frac{|p-m|^3 + |m|^3}{k^2}\right)\right)
    \left\{\left(\left.\left(f_{p-m} + \left\langle \frac{m}{2k}, \nabla f_{p-m}\right\rangle\right)\right|_{x + \frac{p}{2k}} + O\left(\frac{|m|^2}{k^2}\right) \right) \right. \\
    &\left. \qquad \cdot\left(\left.\left(g_{m} - \left\langle \frac{p-m}{2k}, \nabla g_{m}\right\rangle\right)\right|_{x + \frac{p}{2k}} + O\left(\frac{|p-m|^2}{k^2}\right) \right) \right\} +(\mbox{error}). 
\end{align*}
Here the last equation uses the Taylor expansion of $f_{p-m}$ and $g_m$. 
The terms in the sum for $K_{\phi^k(f)\phi^k(g)}$ is taken for $m$ with $x + m/k \in U$, and the term (error) comes from the contributions from those points $b \in B_k \setminus U$ which are both close to $x$ and $x + p/k$, and we see below that this error term is indeed negligible as $k \to \infty$. 
So, at least informally, we see that
\begin{align*}
    K_{\phi^k(f)\phi^k(g)}\left(x + \frac{p}{k}, x\right) =K_{fg}\left(x +\frac{p}{k}, x\right) + \frac{-\sqrt{-1}}{2 k}K_{\{f, g\}}\left(x +\frac{p}{k}, x\right) + O\left(\frac{1}{k^2}\right). 
\end{align*}
So, what we have to do is to give appropriate estimates of the operator norms (not only matrix coefficients) of the error terms. 
Since the proof is long, we give a detailded proof in the Appendix.
\end{proof}

From this, we conclude that the maps $\{\phi_{H, \mathcal{U}}^k\}_{k \in \N}$ solves Problem \ref{prob_intro}. 
\begin{thm}\label{thm_general}
Let $(X^{2n}, \omega, L, \nabla, h)$ be a prequantized symplectic manifold equipped with a proper Lagrangian fiber bundle $\mu \colon X \to B$ with connected fibers. 
Assume we are given a horizontal distribution $H \subset TX$ satisfying the condition in (H) and an open covering $\mathcal{U}$ satisfying the conditions in (U). 
Then, the family of adjoint-preserving linear maps $\{\phi_{H, \mathcal{U}}^k\}_{k \in \N}$, defined in Definition \ref{def_rep_gen.y}, is a strict deformation quantization for $(X, \omega)$. 
\end{thm}
\begin{proof}
The condition (1) in Definition \ref{def_strict_DQ} can be proved in a similar way as in the proof of Proposition \ref{prop_model_norm.y} and Theorem \ref{thm_main_v2.y}, so we leave the details to the reader. 
The condition (2) follows from Theorem \ref{thm_main_v2.y}. 
\end{proof}

As we remarked earlier, the choice of open covering $\mathcal{U}$ for $B$ is not essential in our construction, as follows. 
\begin{prop}\label{prop_open_cov.y}
Let $(X^{2n}, \omega, L, \nabla, h)$ be a prequantized symplectic manifold equipped with a proper Lagrangian fiber bundle $\mu \colon X \to B$ with connected fibers. 
Assume we are given a horizontal distribution $H \subset TX$ satisfying the condition in (H) and two choices of open coverings $\mathcal{U}$ and $\mathcal{V}$ satisfying the conditions in (U). 
Then, for any function $f \in C_c^\infty(X)$, we have
\begin{align*}
    \|\phi^k_{H, \mathcal{U}}(f) - \phi^k_{H, \mathcal{V}}(f)\|
    =O(k^{-N})
\end{align*}
for all $N \in \N$. 
\end{prop}
\begin{proof}
This follows easily from the fact that the Fourier coefficient of a smooth function on $T^n$ are rapidly decreasing, as in \eqref{eq_bd_fourier.y}. 
Indeed, using this fact, we can show that matrix elements of $\phi^k_{H, \mathcal{U}}(f) - \phi^k_{H, \mathcal{V}}(f)$ are of $O(k^{-N})$ for any $N \in \N$ by a similar estimate as in the proof of Theorem \ref{thm_main_v2.y}, and the result follows. 
We leave the details to the reader. 
\end{proof}

\begin{rem}\label{rem_open_cov.y}
As should be obvious from the proof of Theorem \ref{thm_main_v2.y} in the Appendix, the assumption (U3) is not essential. 
Indeed, we may drop this condition. 
We can define $\phi^k_{H, \mathcal{U}}$ in the same way, and show that Theorem \ref{thm_main_v2.y} extends to this case. 
We put the condition (U3) only because it simplifies the proof of Theorem \ref{thm_main_v2.y}. 
Since Proposition \ref{prop_open_cov.y} also extends to this general case, we lose nothing by requiring the condition (U3). 

In particular, in our construction of $\phi^k$ in the model case $\R^n \times T^n$ in subsection \ref{subsec_model.y}, we used such $\mathcal{U}$, namely we set $\mathcal{U} = \{\R^n\}$ (and set $H = T\R^n$). 
Later in section \ref{sec_relation.y}, we again use this trivial non-relatively compact covering for $\R^n$ and consider deformation quantizations (corresponding to non-trivial $H$). 
\end{rem}

\section{Star products}\label{sec_star}
In this section we analyze the star products induced by the deformation quantization defined in Section \ref{sec_construction.y}. 
Our construction is given by explicit formula locally, so in principle we can compute higher terms of star products in action-angle coordinates. 
In the case where the horizontal distribution $H$ is Lagrangian and integrable, we show in Theorem \ref{thm_star_integrable} that the star product obtained from our strict deformation quantization conincides with the star product given by the Fedosov's construction \cite{Fedosov1994}. 
In general cases, by an explicite computation we check this coincidence up to second order term in Theorem \ref{thm_star_gen.y}. 

We consider the settings in subsection \ref{subsec_construction_general.y}. 
We are given a prequantized closed symplectic manifold $(X^{2n}, \omega, L, \nabla)$ and a proper Lagrangian fiber bundle structure $\mu \colon X \to B$ with connected fibers. 
We fix a horizontal distribution $H \subset TX$ and a finite open covering $\mathcal{U}$ of $B$ satisfying conditions in (H) and (U) in subsection \ref{subsec_construction_general.y}. 
Let us consider the deformation quantization $\{\phi^k_{H, \mathcal{U}}\}_k$ defined in Definition \ref{def_rep_gen.y} from these datum. 

The horizontal distribution $H$ associates a torsion-free symplectic connection $\nabla^{TX, H}$ on $TX$, as follows. 
As a first step, we define a connection $\widetilde{\nabla}^{TX, H}$ on $TX$, which does not necessarily preserve the symplectic form and possibly has torsion.  
By the identification $\mu^* TB \simeq H$ and the flat connection on $TB$, we get the pullback connection on $H$. 
Moreover, the vertical tangent bundle $\ker d\mu$ admits the canonical flat connection. 
We define the connection $\widetilde{\nabla}^{TX, H}$ by the direct sum of these two connections, using $TX = H \oplus \ker d\mu$. 

This connection is locally described as follows. 
Let us focus on one open set $U \in \mathcal{U}$. 
We fix an action-angle coordinate on $X_U$, and express $H$ on $X_U$ as in \eqref{eq_def_A} using a locally defined $\R^n$-valued one-form $A \in C^\infty(U; T^*U \otimes \R^n)$.
Denote the Christoffel symbol of $\widetilde{\nabla}^{TX, H}$ with respect to the local frame $(\del_{x_1}, \cdots, \del_{x_n}, \del_{\theta_1}, \cdots, \del_{\theta_n}  )$ by $\widetilde{\Gamma}_{\cdot \cdot}^\cdot$. 
\begin{lem}
We have
\begin{align*}
    \widetilde{\Gamma}_{x_l x_j}^{\theta_i} = -\del_{x_l}A_j^i, 
\end{align*}
and $\widetilde{\Gamma}_{\cdot \cdot}^\cdot = 0$ for all other components.  
\end{lem}
\begin{proof}
The horizontal lift of $\del_{x_j} \in C^\infty(U; TU)$ is given by
$\tilde{\del}_{x_j} =\del_{x_j} + A_j^i \del_{\theta_i} $. 
Since we have $\widetilde{\nabla}^{TX, H}_{\del_{x_l}}\tilde{\del}_{x_j} =0$ and $\widetilde{\nabla}^{TX, H}_{\del_{x_l}} \del_{\theta_i} = 0$, we have
\begin{align*}
    \widetilde{\nabla}^{TX, H}_{\del_{x_l}} \del_{x_j} = -\widetilde{\nabla}^{TX, H}_{\del_{x_l}} A_j^i \del_{\theta_i}
    = -\frac{\del A_i^j}{\del x_l} \del_{\theta_i}. 
\end{align*}
The vanishing for other cases are obvious. 
\end{proof}

Note that $\widetilde{\nabla}^{TX, H}$ is symplectic if and only if $\widetilde{\Gamma}_{x_l x_j}^{\theta_i} = \widetilde{\Gamma}_{x_l x_i}^{\theta_j} $, and torsion-free if and only if $\widetilde{\Gamma}_{x_l x_j}^{\theta_i} = \widetilde{\Gamma}_{x_j x_l}^{\theta_i} $, for all $i, j, l$. 
Now we define a torsion-free symplectic connection $\nabla^{TX, H}$ by symmetrization, as follows. 

\begin{defn}\label{def_symp_conn}
  Assume we are given a symplectic manifold $(X, \omega)$ with a proper Lagrangian fiber bundle $\mu \colon X \to B$ as well as a horizontal distribution $H$ satisfying the condition in (H) in subsection \ref{subsec_construction_general.y}. 
  In a locally defined action-angle coordinate chart $X_U \simeq U \times T^n$ as above, we define a connection on $TX_U$ by requiring its Cristoffel symbol $\Gamma_{\cdot, \cdot}^\cdot$ with respect to the local frame $(\del_{x_1}, \cdots, \del_{x_n}, \del_{\theta_1}, \cdots, \del_{\theta_n}  )$ to be
  \begin{align*}
      \Gamma_{x_l x_j}^{\theta_i} 
      =-\frac{1}{6}\left( \del_lA_j^i + \del_lA_i^j +\del_iA_j^l +\del_iA_l^j +\del_jA_l^i +\del_jA_i^l
    \right), 
  \end{align*}
  and $\Gamma_{\cdot \cdot}^\cdot = 0$ for other components. 
  This construction does not depend on the choice of action-angle coordinate, so we get a global torsion-free symplectic connection on $TX$. 
  We define $\nabla^{TX, H}$ to be this connection. 
\end{defn}
Indeed, it is easily checked that the above symmetrization procedure of Cristoffel symbols is compatible with the change of action-angle coordinates. 

In general, for a symplectic manifold $(X, \omega)$, if we fix a torsion-free symplectic connection $\nabla^{TX}$ on $X$, for each closed element $a \in \hbar \Omega^2(X)[[\hbar]]$, Fedosov's construction \cite{Fedosov1994} associates a star product on $C^\infty(X)[[\hbar]]$, denoted by $*_{H, a}$. 
Moreover Nest and Tsygan \cite{NestTsygan1995} showed that the set of equivalence classes of star products is in one-to-one correspondence with the set of equivalence classes of formal deformations of the symplectic structure, $\hbar H_{dR}^2(X)[[\hbar ]]$. 
In particular if we set $a = 0$, we get a star product $*_{\nabla, 0}$, which is canonically associated to the torsion-free symplectic connection.  
Applying this to our case, we have a canonical choice of star product corresponding to the connection $\nabla^{TX, H}$ and $0 \in \hbar \Omega^2(M)[[\hbar]]$, denoted by $*_{H, 0} := *_{\nabla^{TX, H}, 0}$. 
We denote by $\mathcal{C}_j^{H, 0}(f, g)$ the $j$-th coefficient of $\hbar$ in $f *_{H, 0} g$ for $f, g \in C^\infty(X)$. 
i.e., we have
\begin{align*}
    f *_{H, 0} g = \sum_{j = 0}^\infty \mathcal{C}_j^{H, 0}(f, g) \hbar^j. 
\end{align*}

First, we consider the symplest case when the horizontal distribution $H$ is Lagrangian and integrable. 
In this case $\widetilde{\nabla}^{TX, H}$ is already torsion-free and symplectic, and we have $\widetilde{\nabla}^{TX, H} = {\nabla}^{TX, H}$. 
\begin{thm}\label{thm_star_integrable}
Assume the horizontal distribution $H$ is Lagrangian and integrable. 
Then we have, for all $f , g \in C_c^\infty(X)$ and $l \in \N$, 
\begin{align*}
        \left\| \phi^k_{H, \mathcal{U}}(f) \phi^k_{H, \mathcal{U}}(g) -
        \sum_{j = 0}^l \left( \frac{-\sqrt{-1}}{k}\right)^j \phi^k_{H, \mathcal{U}}\left(\mathcal{C}_j^{H, 0}(f, g)\right)
        \right\| =O\left(\frac{1}{k^{l+1}}\right)
\end{align*}
as $k \to \infty$. 
\end{thm}

\begin{proof}
We work on an element $U \in \mathcal{U}$. 
Since $H$ is Lagrangian and integrable, we can choose the action-angle coordinate on $X_U$ so that $H = \mathrm{Span}\{\del_{x_i}\}_i$. 
Using this coordinate and regarding $X_U \subset \R^n \times T^n$, the star product $*_{H, 0}$ coincides with the standard Moyal-Weyl star product $*_{std}$ (see \eqref{eq_std_MW.y}). 
Moreover, our deformation quantization coincides with the one constructed in the ''model case" $(\R^n \times T^n, {}^t\!dx\wedge d\theta)$ in Subsection \ref{subsec_model.y}, modulo contribution from the terms coming from outside $U$. 
So the result essentially follows from Proposition \ref{prop_model_star.y}. 
We need to show that the error terms coming from  outside $U$ is $O(k^{-N})$ for any $N \in \N$, and this is done in the similar way as in the proof of Theorem \ref{thm_main_v2.y}. 
\end{proof}

Next we turn to the general case, where $H$ is not necessarily symplectic or integrable.  
Also in this case, we are able to show that our strict deformation quantization induces a star product, denoted by $\star_{H}$, and this star product coincides with the above star product $*_{H, 0}$ up to order two in $\hbar$. 

From now on we compute the order-two term of the expected star product. 
Let us focus on one element $U \in \mathcal{U}$ and use the local notations as before. 
Using the derivatives of $A_j^i$, we can explicitely compute the complex phase appearing in \eqref{eq_U(1).y}. 

Fix $x , x + m/k\in U \cap \frac{\Z^n}{k}$. 
We use $\alpha \colon \left[-\frac{|m|}{2k}, \frac{|m|}{2k}\right] \to T^n$ defined in the proof of Lemma \ref{lem_arg.y}, constructed from the horizontal lift of the line segment $[x, x+m/k]$ in $U$ with $\alpha(0) = 0$. 
We regard $\bm{u} := m/|m| \in C^\infty(U; TU)$, and $\alpha$ satisfies
\begin{align*}
    \alpha'(t) &= A(\bm{u})|_{t\bm{u} + x + m/(2k)} \\
    &= \frac{1}{|m|}\sum_j m_j A_j|_{t \bm{u} + x + m/(2k)}, \\
    \alpha''(t) &= \frac{1}{|m|^2} \sum_{j,l}m_jm_l \left.\frac{\del A_j}{\del x_l}\right|_{t \bm{u} + x + m/(2k)}. 
\end{align*}
We compute the phase term in \eqref{eq_phase.y} as, 
\begin{align*}
     &-  \left\langle \frac{m}{2k}, \alpha\left(\frac{|m|}{2k}\right)+\alpha\left(-\frac{|m|}{2k}\right) \right\rangle
   + \int_{-|m|/2k}^{|m|/2k} \left\langle s\bm{u} , \alpha'(s) \right\rangle ds \\
   &= -  \left\langle \bm{u}, \alpha''(0) \right\rangle \left( \frac{|m|}{2k}\right)^3 + \int_{-|m|/2k}^{|m|/2k} \left\langle s\bm{u} , \alpha''(0)s \right\rangle ds + O\left( \frac{|m|^4}{k^4}\right) \\
   &= - \frac{|m|^3}{24k^3}\left\langle \bm{u}, \alpha''(0) \right\rangle  + O\left( \frac{|m|^4}{k^4}\right). 
\end{align*}
So we get, for a function $f \in C_c^\infty(X)$, 
\begin{align}\label{eq_phase_higher.y}
    K_f^k(x + \frac{m}{k}, x) / f_m(x + \frac{m}{2k})
    &= \exp \left(\sqrt{-1}\left\{-\frac{|m|^3}{24k^2}\langle \bm{u}, \alpha''(0) \rangle  + O\left( \frac{|m|^4}{k^3}\right)\right\}\right) \\
    &= \exp \left(\sqrt{-1}\left\{-\frac{1}{24k^2} \sum_{i, j, l}m_im_jm_l \del_{l}A_j^i|_{x + m/(2k)}  + O\left( \frac{|m|^4}{k^3}\right)\right\}\right) \notag \\
    &= 1 - \frac{\sqrt{-1}}{24k^2} \sum_{i, j, l}m_im_jm_l \del_{l}A_j^i|_{x + m/(2k)}  + O\left( \frac{|m|^4}{k^3}\right). \notag
\end{align}
Here we denoted $\del_lA_j^i := \frac{\del A_j^i}{\del x_l}$. 
We define a symmetric tree tensor $\Theta \in C^\infty(U; S^3(T^*U))$ as, 
\begin{align}\label{eq_three_tensor.y}
    \Theta &:= \sum_{i , j, l} \Theta_{ijl} dx_i \otimes dx_j\otimes dx_l, \\
    \Theta_{ijl} &:=\frac{1}{6}\left( \del_lA_j^i + \del_lA_i^j +\del_iA_j^l +\del_iA_l^j +\del_jA_l^i +\del_jA_i^l
    \right). \notag
\end{align} 
and we also denote, for $v \in C^\infty(U; TU)$, 
\begin{align}\label{eq_three_tensor_2.y}
    \Theta(v) := \Theta (v \otimes v \otimes v). 
\end{align}

Suppose we are given two functions $f, g\in C_c^\infty(X)$. 
By \eqref{eq_phase_higher.y} we have
\begin{align*}
    K_{fg}^k\left(x + \frac{p}{k}, x\right) &= \left( 1 - \frac{\sqrt{-1}}{24k^2}\Theta(p)|_{x + \frac{p}{2k}} \right)
    \sum_{m\in\Z^n} (f_{p-m}g_m)|_{x + \frac{p}{2k}} + O(k^{-3}), \\
    K_{\{f,g\}}^k\left(x + \frac{p}{k}, x\right)&= \left(  1 -\frac{\sqrt{-1}}{24k^2}\Theta(p)|_{x + \frac{p}{2k}}\right)\{f, g\}_p|_{x + \frac{p}{2k}} + O(k^{-3}), 
\end{align*}
Since we have
\begin{align*}
    &K_f^k\left(x + \frac{p}{k}, x + \frac{m}{k}\right)K_g^k\left(x + \frac{m}{k}, x\right)\\
    &= \left( 1 -\frac{\sqrt{-1}}{24k^2}\left\{\Theta(p-m)\left(x + \frac{p+m}{2k}\right) + \Theta(m)\left(x + \frac{m}{2k}\right)\right\}\right)f_{p-m}\left(x + \frac{p+m}{2k}\right)g_m\left(x + \frac{m}{2k}\right) \\
    & \qquad + O(k^{-3})\\
    &= \left( 1 -\frac{\sqrt{-1}}{24k^2}\left\{\Theta(p-m)+ \Theta(m)\right\}|_{x + \frac{p}{2k}}\right)(f_{p-m}g_m)|_{x + \frac{m}{2k}} + O(k^{-3})
\end{align*}
and
\begin{align*}
    &f_{p-m}\left(x + \frac{p+m}{2k}\right) 
    = \left.\left\{f_{p-m} + \left\langle \frac{m}{2k}, \nabla f_{p-m}\right\rangle 
    + \frac{1}{2}{}^t\!\left(\frac{m}{2k} \right)H(f_{p-m})\left(\frac{m}{2k} \right)\right\}\right|_{x + \frac{p}{2k}} 
    + O\left(\frac{|m|^3}{k^3}\right) , \\
    &g_m\left(x + \frac{m}{2k}\right)
    = \left. \left\{ g_m - \left\langle \frac{p-m}{2k}, \nabla g_m\right\rangle
    +\frac{1}{2}{}^t\!\left(\frac{p-m}{2k} \right)H(g_m)\left(\frac{p-m}{2k} \right)\right\} \right|_{x + \frac{p}{2k}} + O\left(\frac{|p-m|^3}{k^3}\right). 
\end{align*}
If we denote by $K_{err}^k(\cdot, \cdot)$ the matrix coefficients of the operator
$\phi^k_{H, \mathcal{U}}(f)\phi^k_{H, \mathcal{U}}(g) - \phi^k_{H, \mathcal{U}}(fg) + \frac{\sqrt{-1}}{2 k}\phi^k_{H, \mathcal{U}}(\{f, g\})$, 
we get
\begin{align*}
    & K_{err}^k\left(x + \frac{p}{2k}, x\right) \\
    &= \frac{1}{k^2}\sum_{m}\left\{- \frac{\sqrt{-1}}{24} \{\Theta(p-m)+ \Theta(m) - \Theta(p) \}f_{p-m}g_m \right.\\
&   \quad \left. \left. + \frac{1}{8}\{-2\langle m, \nabla f_{p-m} \rangle \cdot \langle p-m, \nabla g_{m} \rangle
    + f_{p-m} \cdot {}^t\! (p-m) H(g_m) (p-m)
    + {}^t\! m H(f_{p-m}) m \cdot g_m\}\right\}\right|_{x + \frac{p}{2k}}\\
    &\quad + O(k^{-3}).  \\
    &= \frac{-1}{k^2}\left.\left\{\sum_{m} \frac{\sqrt{-1}}{24} \{\Theta(p-m)+ \Theta(m) - \Theta(p) \}f_{p-m}g_m
    + \left( \mathcal{C}_2^{\mathrm{std}}(f, g) \right)_{p}
    \right\}\right|_{x + \frac{p}{2k}} + O(k^{-3}), 
\end{align*}
where $\mathcal{C}_2^{\mathrm{std}}(f, g) \in C^\infty(U \times T^n)$ is the coefficient of $\hbar$ in the standard Moyal-Weyl star product on $\R^n \times T^n$, see \eqref{eq_std_MW.y}. 
Since we have
\begin{align*}
    \{\Theta(p-m)+ \Theta(m) - \Theta(p)\}f_{p-m}g_m
    =-3 \sqrt{-1}\sum_{i, j, l} \Theta_{ijl}\left( \frac{\del f_{p-m} }{\del \theta_i}\frac{\del^2 g_m}{\del \theta_j \del \theta_l}  +\frac{\del^2 f_{p-m} }{\del \theta_j \del \theta_l} \frac{\del g_m}{\del \theta_i} \right), 
\end{align*}
we see that the second coefficient in the star product induced by the strict deformation quantization $\{\phi^k_{H, \mathcal{U}}\}_k$ should be given by 
\begin{align*}
    (f, g) \mapsto \frac{1}{8}\sum_{i, j, l} \Theta_{ijl}\left( \frac{\del f }{\del \theta_i}\frac{\del^2 g}{\del \theta_j \del \theta_l}  +\frac{\del^2 f }{\del \theta_j \del \theta_l} \frac{\del g}{\del \theta_i} \right)  + \mathcal{C}_2^{\mathrm{std}}(f, g).
\end{align*}
It is clear that we can continue this Taylor expansion with respect to $k^{-1}$ for higher orders, and get the higher coefficient of the star product recursively. 
Summerizing, we get the followings. 

\begin{thm}\label{thm_star_gen.y}
Let $(X^{2n}, \omega, L, \nabla)$ be a prequantized symplectic manifold equipped with a proper Lagrangian fiber bundle $\mu \colon X \to B$ with connected fibers. 
Assume we are given a horizontal distribution $H$ satisfying the condition (H) in \ref{subsec_construction_general.y}. 
Then, there exists a unique star product $\star_H$ on $C^\infty(X)[[\hbar]]$ satisfying the followings. 
\begin{enumerate}
    \item If we denote by $\mathcal{C}_j^H(f, g)$ the $j$-th coefficient of $\hbar$
in $f \star_{H} g$ for $f, g \in C^\infty(X)$. 
Then we have, for all $f , g \in C_c^\infty(X)$ and $l \in \N$, 
\begin{align*}
        \left\| \phi^k_{H, \mathcal{U}}(f) \phi^k_{H, \mathcal{U}}(g) -
        \sum_{j = 0}^l \left( \frac{-\sqrt{-1}}{k}\right)^j \phi^k_{H, \mathcal{U}}\left(\mathcal{C}_j^{H}(f, g)\right)
        \right\| =O\left(\frac{1}{k^{l+1}}\right)
\end{align*}
as $k \to \infty$. 
Here $\mathcal{U}$ is any choice of open covering of $B$ satisfying the conditions in (U) in subsection \ref{subsec_construction_general.y}. 
\item For $h_1, h_2 \in C^\infty(B)$, we have
\begin{align*}
    \mu^*h_1 \star_H \mu^*h_2 = \mu^* (h_1h_2). 
\end{align*}
In other words, the commutative algebra $C^\infty(B)$ canonically embeds into $(C^\infty(X)[[\hbar]], \star_H)$. 

\item 
Up to order $2$, this star product $\star_H$ coincides with the Fedosov's star product $*_{H, 0}$, corresponding to the connection $\nabla^{TX ,H}$ defined in Definition \ref{def_symp_conn} and $0 \in \hbar \Omega^2(X)[[\hbar]]$, i.e., we have
\begin{align*}
    \mathcal{C}_j^{H}(f, g) = \mathcal{C}_j^{H, 0}(f, g) \mbox{ for }j = 0, 1, 2. 
\end{align*}
The second coefficient is explicitely given by, choosing a local action-angle coordinate, 
\begin{align}\label{eq_star_second}
    \mathcal{C}_2^{H}(f, g) = \frac{1}{8}\sum_{i, j, l} \Theta_{ijl}\left( \frac{\del f }{\del \theta_i}\frac{\del^2 g}{\del \theta_j \del \theta_l}  +\frac{\del^2 f }{\del \theta_j \del \theta_l} \frac{\del g}{\del \theta_i} \right)  + \mathcal{C}_2^{\mathrm{std}}(f, g). 
\end{align}
Here $\mathcal{C}_2^{\mathrm{std}}$ is the second term in the standard Moyal-Weyl star product on $\R^n \times T^n$, and the locally defined symmetric three tensor $\Theta$ is defined in \eqref{eq_three_tensor.y}. 
\end{enumerate}

\end{thm}

\begin{proof}
For (1), in addition to the above argument, we must estimate the operator norms of error terms. 
It is done essentially in the same way as the proof of Theorem \ref{thm_main_v2.y}. 
(2) follows from the fact that for $h \in C^\infty_c(B)$,  $\phi^k_{H, \mathcal{U}}(\mu^*h)$ is a multiplication operator by $h$, so we have $\phi^k_{H, \mathcal{U}}(\mu^*h_1)\phi^k_{H, \mathcal{U}}(\mu^*h_2) = \phi^k_{H, \mathcal{U}}(\mu^*(h_1h_2))$. 

For (3), the formula \eqref{eq_star_second} follows from the above computations. 
To check the desired coincidence, we use the formula for second order term in $*_{\nabla, 0}$ for torsion-free symplectic connection $\nabla$ on $TX$ (see \cite[Proposition 2.13]{FutakiGravy2018})
\begin{align*}
    f *_{\nabla, 0} g  = fg + \frac{\hbar}{2} \{f, g\}  
    + \frac{\hbar^2}{8}\omega^{ij} \omega^{kl} \nabla^2_{ik}f \nabla^2_{jl}g + O(\hbar^3), 
\end{align*}
where $\nabla_{XY}^2 f := (XY - \nabla_XY)f$ is the second covariant derivative, and $(\omega^{ij})_{ij}$ is the inverse matrix of the coefficients in the symplectic form $\omega = \frac{1}{2}\omega_{ij}dx_i \wedge dx_j$. 
Applying this to our case $\nabla = \nabla^{TX, H}$, by Definition \ref{def_symp_conn}, the only nontrivial contribution from the covariant derivative is the terms
\begin{align*}
    &\frac{-\hbar^2}{8}\left\{\omega^{x_l \theta_l} \omega^{x_j \theta_j} (\nabla_{x_l}\del_{x_j}) f \cdot \del_{\theta_l}\del_{\theta_j}g
     + \omega^{\theta_l x_l} \omega^{\theta_j x_j}\del_{\theta_l}\del_{\theta_j}f \cdot (\nabla_{x_l}\del_{x_j}) g 
    \right\} \\
    &= \frac{\hbar^2}{8} 
    \Theta_{ijl}\left\{\del_{\theta_i} f \cdot \del_{\theta_l}\del_{\theta_j}g
    + \del_{\theta_l}\del_{\theta_j}f \cdot \del_{\theta_i}g
    \right\}, 
\end{align*}
so we see that $\mathcal{C}_2^{H, 0}(f, g)$ is also given by the right hand side of \eqref{eq_star_second}, thus we get the result.
\end{proof}

\section{The relation with Berezin-Toeplitz quantization}\label{sec_relation.y}
In this section, we explain the relation between Berezin-Toeplitz quantization and our quantizations. 
Here we restrict our attention to the case of $\R^n \times T^n$ (subsection \ref{subsec_RnTn_relation.y}) and abelian varieties (subsection \ref{subsec_abelian_relation}) with translation invariant complex structures. 
In those cases we have a natural isomorphism between quantum Hilbert spaces using theta basis for $L^2$-holomorphic sections on $L^k$. 
We show in Theorem \ref{prop_rel_dq.y} and Theorem \ref{prop_rel_dq_abelian.y} that, as $k \to \infty$, the operator norm of the difference between our deformation quantization and the Berezin-Toeplitz deformation quantization converges to zero in both cases. 

First we recall the definition of Berezin-Toeplitz deformation quantization (\cite{BMS1994}, \cite{MaMarinescu2008}). 
Let $(X, \omega, J)$ be a symplectic manifold equipped with a compatible complex structure, and $(L, \nabla)$ be a prequantizing line bundle. 
Then the quantum Hilbert space by this K\aaa hler polarization is given by the spaces of $L^2$-holomorphic sections $\{L^2H^0(X_J; L^k)\}_{k \in \N}$. 
For each $k \in \N$, let us denote the orthogonal projection for the subspace $L^2H^0(X_J; L^k)\subset L^2(X_J; L^k)$ by $\Pi_k$. 
\begin{defn}\label{def_BTDQ.y}
  In the above settings, for each $k \in \N$, define a linear operator $T^k \colon C_c^\infty(X) \to \mathbb{B}\left(L^2 H^0(X_J; L^k)\right)$ by
  \begin{align*}
      T^k(f) := \Pi_k M_f \Pi_k^* \colon L^2H^0(X_J; L^k) \to  L^2H^0(X_J; L^k), 
  \end{align*}
  for $f \in C_c^\infty(X)$. 
  Here $M_f$ denotes the multiplication operator by $f$. 
\end{defn}

This sequence has the correct semiclassical behavior in the case where $X$ is compact, as shown by Bordemann, Meinrenken and Schlichenmaier \cite{BMS1994}. 
\begin{fact}[\cite{BMS1994}]
If $(X, \omega, J)$ is a compact K\aaa hler manifold, above sequence $\{T^k\}_k$ is a strict deformation quantization of $C^\infty(X)$, called {\it Berezin-Toeplitz deformation quantization}. 
\end{fact}
This fact has been generalized in various ways by Ma and Marinescu \cite{MaMarinescu2008}, in particular to certain classes of non-compact K\aaa hler manifolds and orbifolds. 
Their result includes the case of $\R^n \times T^n$ with translation invariant K\aaa hler structure, which is of our interest in the following subsection \ref{subsec_RnTn_relation.y}. 

\subsection{On $\R^n \times T^n$. }\label{subsec_RnTn_relation.y}
In this subsection, we explain the convergence in the case of translation invariant K\aaa hler quantizations on $\R^n \times T^n$. 

First we explain our convention on compatible almost complex structures on $\R^n \times T^n$. 
Let $\mathbb{H}^n := \{\Omega \in M_n(\C) \ | \ \Omega = {}^t\! \Omega, \ \mathrm{Im} \Omega \mbox{ is positive definite } \}$ be the Siegel upper half space. 
Then, if we have an $\mathbb{H}^n$-valued function $\Omega \in C^\infty(\R^n\times T^n; \mathbb{H}^n)$, we get an almost complex structure on $\R^n \times T^n$ by
\begin{align*}
    T^{0, 1}_{(x, \theta)}(\R^n \times T^n) = \mathrm{Span}_{\C}\left\{
    \frac{\del}{\del x_i} + \Omega_{ij}(x, \theta)\frac{\del}{\del \theta_j}
    \right\}_{i = 1}^n. 
\end{align*}
This is compatible with $\omega = {}^t\!dx\wedge d\theta$. 

In this subsection we only consider translation invariant complex structures, i.e., the case where $\Omega$ is constant. 
We denote this complex structure by $J_\Omega$. 

Let us consider the prequantum line bundle $(L, \nabla) = (\underline{\C}, d - \sqrt{-1}{}^t\!xd\theta)$. 
If we have a section $s \in C^\infty(\R^n \times T^n; L^k)$, we denote its Fourier expansions by
\begin{align*}
    s(x, \theta) = \sum_{m \in \Z^n}s_m(x) e^{\sqrt{-1}\langle m, \theta \rangle}. 
\end{align*}
Let $k$ be a positive integer. 
It is easy to see that an orthonormal basis $\{\Psi^k_{\Omega, \frac{l}{k}}\}_{l \in \Z^n}$ of $L^2$-holomorphic sections on $L^k$ is given by
\begin{align}\label{eq_Psi_model.y}
    (\Psi^k_{\Omega, \frac{l}{k}})_m = \delta_{l, m}(2\pi)^{-n/2}a_{k, \mathrm{Im}\Omega}\exp\left(\sqrt{-1}k/2 \ {}^t\! \left(x -l/k\right)\Omega \left(x-l/k\right)\right), 
\end{align}
where $a_{k, \mathrm{Im}\Omega}> 0$ is the normalization constant given by
\begin{align*}
    (a_{k, \mathrm{Im}\Omega})^{-2} = 
    \int_{\R^n}\exp \left(-k {}^t\!x (\mathrm{Im}\Omega) x\right) dx. 
\end{align*}

We can write explicitely the Berezin-Toeplitz deformation quantization using this basis as follows. 
\begin{lem}\label{lem_BT_model.y}
For a function $f \in  C_c^\infty(\R^n\times T^n)$, we write the Fourier expansion of $f$ as $f = \sum_{m \in \Z^n}f_m(x)e^{\sqrt{-1}\langle m, \theta \rangle}$. 
Then we have
\begin{align*}
    &\langle \Psi_{\Omega, b}^k, T^k(f) \Psi_{\Omega, c}^k \rangle \\
    &= (a_{k, \mathrm{Im}\Omega})^{2}\int_{x \in \R^n} f_{k(b-c)}(x) \exp (\sqrt{-1}k/2 \{-{}^t\!(x - b)\overline{\Omega} (x - b) + {}^t\!(x - c){\Omega} (x - c)\}) dx. 
\end{align*}
\end{lem}
The proof is straightforward. 

Write $\Omega = P + \sqrt{-1}Q$, $P, Q \in M_n(\R)$. 
We consider the splitting $T(\R^n \times T^n) = H_P \oplus \ker d\mu$ given by
\begin{align}\label{eq_horizontal.y}
    H_{P}|_{(x, \theta)} = \mathrm{Span}\left\{ \frac{\del}{\del x_i} + P_{ij}\frac{\del}{\del \theta_j}
    \right\}_{1 \le i \le n}. 
\end{align}
for all $(x, \theta) \in \R^n \times T^n$. 

Recall that the quantum Hilbert space by the real polarization $\mu$ is given by
\begin{align*}
    \mathcal{H}_k = \oplus_{b \in \frac{\Z^n}{k}} H^0(X_b; L^k\otimes |\Lambda|^{1/2}X_b). 
\end{align*}
and we use the orthonormal basis $\{\psi^k_b\}_{b \in \frac{\Z^n}{k}}$, $\psi^k_b = e^{\sqrt{-1}k\langle b, \theta \rangle}\sqrt{d'\theta}$, as in \eqref{eq_basis_model.y}. 
We consider the strict deformation quantization in Definition \ref{def_rep_gen.y} associated to the horizontal distribution $H_P$ and the trivial covering $\mathcal{U} = \{\R^n\}$ (see Remark \ref{rem_open_cov.y}), denoted by $\{\phi^k_{H_P}\}_k$. 

For each $k \in \N$, we get the canonical isomorphism between quantum Hilbert spaces, 
\begin{align}\label{isom_hilb_model.y}
    \mathcal{H}_k \simeq L^2H^0(X_{J_\Omega}; L^k), \ \psi^k_b \mapsto \Psi^k_{\Omega, b}, 
\end{align}
for each $b \in \frac{\Z^n}{k}$. 
Using this isomorphism, we can describe the relation between Berezin-Toeplitz deformation quantization and our deformation quantization as follows.

\begin{thm}\label{prop_rel_dq.y}
Consider the complex structure on $\R^n \times T^n$ associated with $\Omega = P + \sqrt{-1}Q$. 
For all $f \in C_c^\infty(\R^n \times T^n)$, we have
\begin{align*}
  \lim_{k \to \infty} \| \phi^k_{H_P}(f) - T^k(f)\| =  0. 
\end{align*}
Here we use the isomorphism of Hilbert spaces \eqref{isom_hilb_model.y}. 
\end{thm}
\begin{proof}
By the coordinate change
\begin{align*}
    \R^n \times T^n \to \R^n \times T^n, \ (x, \theta) \mapsto (x, -Px + \theta), 
\end{align*}
the Berezin-Toeplitz deformation quantization and our deformation quantization, as well as the isomorphism \eqref{isom_hilb_model.y}, map to the ones for $\Omega = \sqrt{-1}Q$. 
So we may assume $P = 0$. 

First we note the following. 
\begin{lem}\label{lem_delta.y}
There exists a constant $C_Q > 0$ only depending on $Q$ such that, 
for any function $g \in C_c^\infty(\R^n)$, we have
\begin{align*}
    \left| g(0) - (a_{k, Q})^{2} \int_{\R^n} g(x)\exp(-k {}^t\! xQx)dx
    \right|
    \le \frac{C_Q\|\nabla g\|_{C^0}}{\sqrt{k}}. 
\end{align*}
\end{lem}
\begin{proof}
This follows from the estimate
\begin{align*}
    |g(x) - g(0)| \le \|\nabla g\| |x|
\end{align*}
and
\begin{align*}
    (a_{k, Q})^{-2} = \int_{\R^n}  \exp(-k {}^t\! xQx) dx
    = k^{-n/2} \int \exp(- {}^t\! xQx) dx, \\
    \int_{\R^n} |x| \exp(-k {}^t\! xQx) dx
    = k^{-(n+1)/2} \int |x|\exp(- {}^t\! xQx) dx. 
\end{align*}
\end{proof}

\begin{lem}\label{lem_diff_dq.y}
Assume we are given a function $f = \sum_{m \in \Z^n}f_m(x) e^{\sqrt{-1}\langle m, \theta \rangle} \in C^\infty(\R^n \times T^n)$.
For each $N \in \N$, there exists a constant $C_N > 0$, which only depends on $f$ and $Q$, such that, for all $k \in \N$, $p \in \Z^n$ and $x_0 \in \frac{\Z^n}{k}$, we have
\begin{align*}
    |\langle \Psi^k_{\Omega, x_0 + p/k}, T^k(f) \Psi^k_{\Omega, x_0}\rangle - f_{p}(x_0 + p/(2k))| \le \frac{C_N}{\sqrt{k}(1 + |p|)^N}. 
\end{align*}
\end{lem}
\begin{proof}
We have
\begin{align*}
    &-{}^t\!(x - (x_0 + p/k))\overline{\Omega} (x - (x_0 + p/k)) + {}^t\!(x - x_0){\Omega} (x -x_0)\\
    &=  \sqrt{-1}\left\{2{}^t\!(x - (x_0 + p/(2k)))Q (x - (x_0 + p/(2k))) + \frac{1}{2k^2}{}^t\!pQp\right\}. 
\end{align*}
By Lemma \ref{lem_BT_model.y}, we have
\begin{align*}
    &\langle \Psi^k_{\Omega, x_0 + p/k}, T^k(f) \Psi^k_{\Omega, x_0}\rangle \\
    &= (a_{k, Q})^{2}\exp(-(4k)^{-1} {}^t\!pQp) \\
    & \cdot \int_{x \in \R^n} f_{p}(x) 
    \exp\{ - k {}^t\!(x - (x_0 + p/(2k)))Q (x - (x_0 + p/(2k)))\} dx \\
    &= (a_{k, Q})^{2}\exp\left(-\frac{1}{4k} {}^t\!pQp\right)
    \cdot \int_{x \in \R^n} f_{p}(x +(x_0 + p/(2k)) ) 
    \exp(- k {}^t\!xQ x ) dx
\end{align*}
By Lemma \ref{lem_delta.y}, we have
\begin{align*}
  & \left|f_p(x_0 + p/(2k))- (a_{k, Q})^{2}\int_{x \in \R^n} f_{p}(x +(x_0 + p/(2k)) ) 
    \exp(  - k {}^t\!xQ x ) dx\right| \\
   & \le \frac{C_Q \|\nabla f_p\| }{\sqrt{k}}. 
\end{align*}
Since we have $|1 - e^{-y}| \le y$ for any $y \ge 0$, we see 
\begin{align*}
    \left|\langle \Psi^k_{\Omega, x_0 + p/k}, T^k(f) \Psi^k_{\Omega, x_0}\rangle - f_p(x_0 + p/(2k))\right| 
    &\le  \frac{C_Q \|\nabla f_p\| }{\sqrt{k}} + \|f_p\| \cdot \frac{{}^t\!pQp}{4k} \\
    &\le \frac{C_N}{\sqrt{k}(1 + |p|)^{N}}, 
\end{align*}
for some constant $C_{N}$ independent of $k$, $x$, $p$, since the Fourier coefficients are rapidly decreasing. So we get the result. 
\end{proof}

Now we prove Theorem \ref{prop_rel_dq.y}. 
Using the isomorphism \eqref{isom_hilb_model.y}, we regard $T^k(f)$ as an operator on $\mathcal{H}_k$. 
Using Lemma \ref{lem_est_norm.y} and Lemma \ref{lem_diff_dq.y}, we have
\begin{align*}
    \|\phi^k_{H_P}(f) - T^k(f)\|
    &\le \sum_{p \in \Z^n}  \sup_{x \in \frac{\Z^n}{k}} \left\{\left|\langle \Psi^k_{\Omega, x_0 + p/k}, T^k(f) \Psi^k_{\Omega, x_0}\rangle - f_{p}(x_0 + p/(2k))\right|\right\} \\
   & \le \frac{C_{n+1}}{\sqrt{k}} \sum_{p \in \Z^n} \frac{1}{(1 + |p|)^{n+1}}, 
\end{align*}
so we get the result. 

\end{proof}

\subsection{On Abelian varieties}\label{subsec_abelian_relation}
In this subsection we show the relation between Berezin-Toeplitz deformation quantization and our deformation quantization in the case of Abelian varieties. 
For works relating geometric quantization on Abelian varieties by different polarizations, see for example \cite{BMN2010}. 

Let $X = (\R/\Z)^n \times T^n$ be the $2n$-dimensional torus which is obtained by the $\Z^n$-action on $\R^n \times T^n$ considered in the subsection \ref{subsec_RnTn_relation.y}, where $m \in \Z^n$ acts by the symplectomorphism
\begin{align*}
    (x, \theta) \mapsto (x + m, \theta), 
\end{align*}
and consider the induced symplectic structure $\omega = {}^t\! dx \wedge d\theta$ on $X$. 
We get the induced Lagrangian fibration $\mu \colon X \to \R^n/\Z^n$. 
The $\Z^n$-action lifts to an action on the prequantizing line bundle $(\underline{\C}, \nabla = d - \sqrt{-1}{}^t\!xd\theta)$ on $\R^n \times T^n$ by
\begin{align}\label{eq_equivariance_L.y}
    (x, \theta, v) \mapsto (x+m, \theta, e^{\sqrt{-1}\langle m, \theta \rangle} v), 
\end{align}
preserving $\nabla$. 
So we get the induced prequantizing line bundle on $X$, denoted by $(L, \nabla)$. 
In this case, the set of $k$-Bohr-Sommerfeld point is given by $B_k = (\frac{1}{k}\Z / \Z)^n\subset (\R /\Z)^n$. 
A section $s \in C^\infty(X; L^k)$ is identified with a section $\tilde{s} \in C^\infty(\R^n \times T^n; \underline{\C})$ with the periodicity property
\begin{align*}
  \tilde{s}(x + m, \theta) = e^{\sqrt{-1}\langle km, \theta \rangle}\tilde{s}(x, \theta).    
\end{align*}

Fix an element $\Omega \in \mathbb{H}^n$. 
From a translation invariant complex structure $\R^n \times T^n$ given by $\Omega$ as in subsection \ref{subsec_RnTn_relation.y}, we get the induced translation invariant $\omega$-compatible complex structure on $X$, also denoted by $J_\Omega$. 
An orthonormal basis $\{\Theta^k_{\Omega, b}\}_{b \in B_k}$ of $H^0(X_{J_\Omega}; L^k)$ is given by the following formula for its lift $\widetilde{\Theta}^k_{\Omega, b} \in C^\infty(\R^n \times T^n; L^k)$, 
\begin{align}\label{eq_Psi_model_abelian.y}
    \widetilde{\Theta}^k_{\Omega, b}= \sum_{l \in \Z^n, [\frac{l}{k}] = b} \Psi^k_{\Omega, \frac{l}{k}}, 
\end{align}
where $\Psi^k_{\Omega, \frac{l}{k}}$ is the basis of $L^2H^0((\R^n \times T^n)_{J_\Omega}; L^k)$ given in \eqref{eq_Psi_model.y}. 
It is easy to see that this basis coincides with the classical Theta basis (see \cite[Section 2.3]{BMN2010}). 

On the other hand, the orthonormal basis $\{\psi_b^k\}_{b \in \frac{\Z^n}{k}}$ for the quantum Hilbert space by the real polarization on $\R^n \times T^n$ as in \eqref{eq_basis_model.y} induces the orthonormal basis for the quantum Hilbert space $\mathcal{H}_k$ by the real polarization $\mu$ on $X$, since $\{\psi_b^k\}_b$ satisfies the equivariance property with respect to the $\Z^n$-action \eqref{eq_equivariance_L.y}. 
We denote the induced orthonormal basis on $\mathcal{H}_k$ by $\{\vartheta_b^k\}_{b \in B_k}$. 
Using these basis, we get the canonical isomorphism of quantum Hilbert spaces, 
\begin{align}\label{eq_isom_Hilb_abelian.y}
    \mathcal{H}_k \simeq H^0(X_{J_\Omega}; L^k), \ \vartheta^k_b \mapsto \Theta^k_{\Omega_s, b}, 
\end{align}

Corresponding to $\Omega = P + \sqrt{-1}Q$, we consider the horizontal distribution $H_P \subset TX$ induced from \eqref{eq_horizontal.y}.

We have the following relation between Berezin-Toeplitz deformation quantization and our deformation quantization for this case. 
\begin{thm}\label{prop_rel_dq_abelian.y}
Consider the complex structure on $X  = (\R/\Z)^n \times T^n$ associated with $\Omega = P + \sqrt{-1}Q$. 
Choose any open covering $\mathcal{U}$ of $(\R/\Z)^n$ satisfying the conditions in (U) in subsection \ref{subsec_construction_general.y}, and consider the associated strict deformation quantization $\{\phi^k_{H_P, \mathcal{U}}\}_k$. 
For all $f \in C_c^\infty(X)$, we have
\begin{align*}
  \lim_{k \to \infty} \| \phi^k_{H_P, \mathcal{U}}(f) - T^k(f)\| = 0. 
\end{align*}
Here we use the isomorphism of Hilbert spaces \eqref{eq_isom_Hilb_abelian.y}. 
\end{thm}
\begin{proof}
Let us fix a function $f \in C_c^\infty(X)$. 
We denote the lift of $f$ to $\R^n \times T^n$ by $\tilde{f} \in C^\infty(\R^n \times T^n)$. 
Let us denote by $\tilde{T}^k \colon C_c^\infty(\R^n \times T^n) \to \mathbb{B}\left(L^2H^0((\R^n \times T^n)_{J_\Omega}; L^k)\right)$ the Berezin-Toeplitz deformation quantization on $\R^n \times T^n$. 
First we easily observe that, we can define a bounded operator $\tilde{T}^k(\tilde{f}) \in \mathbb{B}\left(L^2H^0((\R^n \times T^n)_{J_\Omega}; L^k)\right)$ by the same formula as in Definition \ref{def_BTDQ.y}, by the periodicity of $\tilde{f}$. 
For $b, c \in B_k = (\frac{1}{k}\Z/\Z)^n$, choose any lift $\bar{b}, \bar{c} \in \frac{\Z^n}{k}$. 
Then we have
\begin{align*}
    \langle \Theta_b^k, T^k(f)\Theta_c^k \rangle
    = \sum_{l \in \Z^n} \langle \Psi_{\bar{b} + l}^k, \tilde{T}^k(\tilde{f}) \Psi_{\bar{c}}^k\rangle. 
\end{align*}
Note that if we denote the Foourier expansion $f(x, \theta) = \sum_m f_m(x) e^{\sqrt{-1}\langle m, \theta \rangle}$, we have
\begin{align*}
    \left| \langle \Psi_{x_0 + \frac{l}{k}}^k, \tilde{T}^k(\tilde{f}) \Psi_{x_0}^k\rangle\right| \le \|f_{l}\|_{C^0}. 
\end{align*}
Combining this and Lemma \ref{lem_diff_dq.y}, as well as the fact that the Fourier coefficients of $f$ are rapidly decreasing, we easily get the result. 
The details are left to the reader. 
\end{proof}

\section{Appendix}

\subsection{A proof of Proposition \ref{prop_model_norm.y}}
In this subsection, we give a detailed proof of Proposition \ref{prop_model_norm.y}. 

\begin{proof}[Proof of Proposition \ref{prop_model_norm.y}]
For simplicity we only prove in the case where $n = 1$; for general $n$ the proof is essentially the same.
Since the maps $\phi^k$ are adjoint-preserving, it is enough to show in the case when $f \in C_c^\infty(X; \R)$. 
In this case, $\phi^k(f)$ is a self-adjoint operator for each $k$. 

First we show that $\sup_{k \to \infty} \|\phi^k(f)\| \le \|f\|_{C^0}$. 
Assume the contrary, and take $\epsilon > 0$ with $\epsilon < \|f\|_{C^0}$ so that 
\begin{align}\label{ass_model_contr.y}
    \sup_{k \to \infty} \|\phi^k(f)\| > \|f\|_{C^0} + \epsilon. 
\end{align}
Let us denote the Fourier expansion of $f$ in the $T$-direction as $\sum_{m \in \Z}f_m(x)e^{ \sqrt{-1} \langle m, \theta \rangle}$. 
Since $f$ is smooth and compactly supported, there exists a positive integer $M \in \N$ such that
\begin{align*}
    \sum_{|m| > M}\|f_m\|_{C^0} < \frac{\epsilon}{100}. 
\end{align*}
Fix such $M$. Let $\tilde{f} \in C_c^\infty(X)$ be the function defined by
\begin{align}\label{eq_def_tildef.y}
    \tilde{f} := \sum_{|m| \le M } f_m(x)e^{ \sqrt{-1} \langle m, \theta \rangle}. 
\end{align}
It is easy to see that
\begin{align}\label{eq_diff_sup.y}
    \|f - \tilde{f} \|_{C^0} < \frac{\epsilon}{100}. 
\end{align}
Moreover, applying Lemma \ref{lem_est_norm.y}, for all $k\in \N$ we have
\begin{align*}
    \|\phi^k(f) - \phi^k(\tilde{f})\| \le  \sum_{|m| > M}\|f_m\|_{C^0} < \frac{\epsilon}{100}.
\end{align*}
So by our assumption \eqref{ass_model_contr.y}, we have
\begin{align}\label{ass_model_contr2.y}
\sup_{k \to \infty} \|\phi^k(\tilde{f})\| > \|f\|_{C^0} + \frac{\epsilon}{2}.
\end{align}

Since $f_m \in C_c^\infty(\R)$ is smooth and compactly supported for each $m$, there exists $\alpha > 0$ such that for all $m \in \Z$ with $|m | \le M$, we have,
\begin{align}\label{eq_unif_conti.y}
    |f_m(x) - f_m (y)| < \frac{\epsilon}{100(2M+1)} \mbox{ for all } |x - y| \le \alpha . 
\end{align}
Let us choose $k_0 \in \N$ such that
\begin{align}\label{eq_def_k_model.y}
    k_0 \ge \frac{100M\|f\|_{C^0}}{\epsilon \alpha}, \mbox{ and }
    \|\phi^{k_0}(\tilde{f})\| \ge \|f\|_{C^0} + \frac{\epsilon}{2}. 
\end{align}
This is possible by \eqref{ass_model_contr2.y}. 
Notice that $\phi^{k_0}(\tilde{f})$ is a self-adjoint operator on $\mathcal{H}_{k_0}$, and by \eqref{eq_def_tildef.y}, there exists a finite interval $[a, b] \subset \R$ such that
\begin{align*}
    \phi^{k_0}(\tilde{f}) = P_{[a, b]}\phi^{k_0}(\tilde{f})P_{[a, b]}.  
\end{align*}
Indeed, it is enough to take $[a, b]$ so that $\mathrm{supp}(f) \subset (a + \frac{m}{2{k_0}}, b - \frac{m}{2{k_0}})$. 
So the operator $\phi^{k_0}(\tilde{f})$ can be regarded as a linear operator on the finite dimensional Hilbert space $\mathcal{H}_{k_0}|_{[a,b]}$. 
In particular, by \eqref{eq_def_k_model.y}, we can take an eigenvalue $\lambda \in \R$ of $\phi^{k_0}(\tilde{f})$ with 
\begin{align}\label{eq_lambda.y}
    |\lambda| \ge \|f\|_{C^0} + \frac{\epsilon}{2}. 
\end{align}
We take a normalized eigenvector $v \in \mathcal{H}_{k_0}$ for $\lambda$, 
\begin{align*}
    \phi^{k_0}(\tilde{f})v = \lambda v, \ \|v\| = 1.  
\end{align*}

\begin{lem}\label{lem_pigeonhole}
    There exists a closed interval $I = [s, t] \subset \R$ such that
    \begin{itemize}
        \item $|I| =( t-s = )  \alpha$, 
        \item If we define the interval $J := [s - \frac{M}{{k_0}}, t + \frac{M}{{k_0}}]$, we have
        \begin{align*}
            \|P_I v\|^2 \ge \frac{{k_0}\alpha}{{k_0}\alpha + 2M} \|P_J v\|^2. 
        \end{align*}
    \end{itemize}
\end{lem}
\begin{proof}
The proof is given by an easy pigeonhole principle argument. 
For each $l \in \Z$, we consider the intervals
\begin{align*}
    I_l := \left[l/{k_0}, l/{k_0}+ \alpha\right], \quad J_l := \left[(l-M)/{k_0}, (l + M)/{k_0} + \alpha \right]. 
\end{align*}
Since $\|v\| = 1$, we have
\begin{align*}
    \sum_{l \in \Z} \|P_{I_l}v\|^2 &= \lfloor {k_0}\alpha + 1 \rfloor, \\ 
    \sum_{l \in \Z} \|P_{J_l}v\|^2 &= \lfloor {k_0}\alpha + 2M + 1 \rfloor. 
\end{align*}
So there must exist an integer $l \in \Z$ with
\begin{align*}
    \|P_{I_l}v\|^2 \ge \frac{{k_0}\alpha}{{k_0}\alpha + 2M} \|P_{J_l} v\|^2. 
\end{align*}
So we can set $I := I_l$ for such $l$ and get the result. 
\end{proof}

Fix an interval $I, J \subset \R$ satisfying the conditions in Lemma \ref{lem_pigeonhole}. 
By \eqref{eq_def_tildef.y} and the definition of $\phi^{k_0}$, we have
\begin{align*}
    P_{x + \frac{m}{{k_0}}} \phi^{k_0}(\tilde{f}) P_x = 0
\mbox{ for all }x \in B_{k_0} \mbox{ and } m \in \Z \mbox{ with }|m| > M. 
\end{align*}
so we get
\begin{align*}
    \lambda P_I v = P_I \phi^{k_0}(\tilde{f}) v =   P_I \phi^{k_0}(\tilde{f})P_J v. 
\end{align*}
From this and Lemma \ref{lem_pigeonhole}, we have
\begin{align*}
    \|P_I \phi^{k_0}(\tilde{f})P_J\| \ge \lambda \frac{\|P_Iv\|}{\|P_Jv\|}
    \ge \lambda \sqrt{\frac{{k_0}\alpha}{{k_0}\alpha + 2M}} 
    \ge \left(\|f\|_{C^0} + \frac{\epsilon}{2}\right)\left(1 + \frac{\epsilon}{50 \|f\|_{C^0}} \right)^{-1/2}, 
\end{align*}
where we used \eqref{eq_lambda.y} and \eqref{eq_def_k_model.y} for the last inequality. 
The last expression is estimated as
\begin{align*}
    \left(\|f\|_{C^0} + \frac{\epsilon}{2}\right)\left(1 + \frac{\epsilon}{50 \|f\|_{C^0}} \right)^{-1/2}
    \ge \left(\|f\|_{C^0} + \frac{\epsilon}{2}\right)\left(1 - \frac{\epsilon}{50 \|f\|_{C^0}} \right)
    > \|f\|_{C^0} + \frac{\epsilon}{4}, 
\end{align*}
since we have assumed that $\epsilon < \|f\|_{C^0}$. 
So we have
\begin{align}\label{eq_firstest_tildef.y}
     \|P_I \phi^{k_0}(\tilde{f})P_J\| >  \|f\|_{C^0} + \frac{\epsilon}{4}. 
\end{align}

Let us denote by $x_0 \in I$ the middle point of the interval $I$. 
We define $F \in C^\infty(T)$ by 
\begin{align*}
    F(\theta) := \tilde{f}(x_0, \theta). 
\end{align*}
We consider the operator $\Phi^{k_0}(F) \in \mathcal{H}_{k_0}$ defined in Lemma \ref{lem_const_model.y}. 

\begin{lem}\label{lem_approx_const_model.y}
We have
\begin{align*}
    \left\|P_I\left( \Phi^{k_0}(F) -  \phi^{k_0}(\tilde{f})\right)P_J\right\| \le \frac{\epsilon}{100}. 
\end{align*}
\end{lem}
\begin{proof}
Recall the definition of $\Phi^{k_0}(F)$ in Lemma \ref{lem_const_model.y}, the construction of $\tilde{f}$ in \eqref{eq_def_tildef.y} and the estimate \eqref{eq_unif_conti.y}. 
Applying the estimate in Lemma \ref{lem_est_norm.y}, 
\begin{align*}
    \|P_I( \Phi^{k_0}(F) -  \phi^{k_0}(\tilde{f}))P_J\|
    \le \sum_{|m| \le M} \sup_{x \in J} |f_m(x) - f_m(x_0)|
    \le \sum_{|m| \le M} \frac{\epsilon}{100(2M+1)}
    \le \frac{\epsilon}{100}. 
\end{align*}
Here we used the inequality $|x - x_0| \le |J|/2 = (\alpha + \frac{M}{{k_0}}) < \alpha$ for $x \in J$. 
This follows from \eqref{eq_def_k_model.y} and our assumption $\epsilon < \|f\|_{C^0}$. 
\end{proof}

By Lemma \ref{lem_const_model.y} and Lemma \ref{lem_approx_const_model.y}, we have
\begin{align*}
    \|P_I \phi^{k_0}(\tilde{f}) P_J\| &\le  \|P_I \Phi^{k_0}(F) P_J\| + \frac{\epsilon}{100}\\
    &\le \|\Phi^{k_0}(F)\| + \frac{\epsilon}{100} \\
    &= \|F\|_{C^0} + \frac{\epsilon}{100}\\
    &\le \|\tilde{f}\|_{C^0} + \frac{\epsilon}{100}. 
\end{align*}
By \eqref{eq_diff_sup.y}, we get
\begin{align*}
    \|P_I \phi^{k_0}(\tilde{f}) P_J\| \le \|f\|_{C^0} +\frac{\epsilon}{50}. 
\end{align*}
This contradicts with the estimate \eqref{eq_firstest_tildef.y}, so we get
\begin{align}\label{eq_proof_model1_sup.y}
    \sup_{k \to \infty} \|\phi^k(f)\| \le \|f\|_{C^0}. 
\end{align}

Next we show $\inf_{k \to \infty} \|\phi^k(f)\| \ge \|f\|_{C^0}$. 
Assume the contrary, and take $\delta > 0$ with $\delta < \|f\|_{C^0}$ such that
\begin{align*}
    \inf_{k \to \infty} \|\phi^k(f)\| < \|f\|_{C^0} - \delta. 
\end{align*}
We take $M' \in \N$, $\hat{f} \in C_c^\infty(X)$ and $\alpha' > 0$ exactly in the same way, this time $\epsilon$ replaced by $\delta$, as in the first half of this proof. 
Namely, we take $M' > 0$ so that
\begin{align*}
    \sum_{|m| > M'}\|f_m\|_{C^0} < \frac{\delta}{100}. 
\end{align*}
We define $\hat{f}$ by
\begin{align}\label{eq_def_hatf.y}
    \hat{f} := \sum_{|m| \le M' } f_m(x)e^{ \sqrt{-1} \langle m, \theta \rangle}. 
\end{align}
This satisfies
\begin{align}
    \|f - \hat{f} \|_{C^0} &< \frac{\delta}{100}, \label{eq_diff_inf.y} \\
  \|\phi^k(f) - \phi^k(\hat{f})\| &< \frac{\delta}{100}, \notag \\
  \inf_{k \to \infty} \|\phi^k(\hat{f})\| &< \|f\|_{C^0} - \frac{\delta}{2}. \label{eq_inf_tildef.y}
\end{align}
Take $\alpha' > 0$ so that
for all $m \in \Z$ with $|m | \le M'$, we have,
\begin{align}\label{eq_def_alpha_inf.y}
    |f_m(x) - f_m (y)| < \frac{\delta}{100(2M'+1)} \mbox{ for all } |x - y| \le \alpha' . 
\end{align}
Let us choose $k' \in \N$ such that
\begin{align}\label{eq_def_k_model_inf.y}
    k' \ge \frac{100M'\|f\|_{C^0}}{\delta \alpha'}, \mbox{ and }
    \|\phi^{k'}(\hat{f})\| \le \|f\|_{C^0} - \frac{\epsilon}{2}. 
\end{align}
This is possible by \eqref{eq_inf_tildef.y}. 

Since we have assumed that $\delta < \|f\|_{C^0}$, by \eqref{eq_def_k_model_inf.y}, we have $k'^{-1} < \alpha'$. 
So by \eqref{eq_def_alpha_inf.y}, we can take $(y_0, \theta_0) \in B_{k'} \times T$ such that
\begin{align}\label{eq_max_value.y}
   | \hat{f}(y_0, \theta_0)| \ge \|\hat{f}\|_{C^0} - \frac{\delta}{100}. 
\end{align}

Let us define the intervals $I', J' \subset \R$ by
\begin{align*}
    I' := [y_0 -\alpha'/2 + M'/k', y_0 +\alpha'/2 - M'/k' ], 
    \quad J' := [y_0 -\alpha'/2, y_0 +\alpha'/2]. 
\end{align*}
Let us define a function $\hat{F} \in C^\infty(T)$ by
\begin{align*}
    \hat{F}(\theta) := \hat{f}(y_0, \theta). 
\end{align*}
We consider the operator $\Phi^{k'}(\hat{F}) \in \mathcal{H}_k$ defined in Lemma \ref{lem_const_model.y}. 
Similarly as Lemma \ref{lem_approx_const_model.y}, we have
\begin{align}\label{eq_approx_const_model_inf.y}
    \|P_{I'}\left( \Phi^{k'}(\hat{F}) -  \phi^{k'}(\hat{f})\right)P_{J'}\| \le \frac{\delta}{100}. 
\end{align}
Let us define $v' \in P_{J'}\mathcal{H}_k$ by
\begin{align*}
    v' := \sum_{|m| \le k' \alpha' / 2} e^{- \sqrt{-1} \langle m,  \theta_0 \rangle} \psi^k_{y_0 + m/k'} \sqrt{d'\theta}. 
\end{align*}
We have $\|v'\|^2 = \#(J' \cap B_{k'})$. 
For $y_0 + \frac{p}{k'} \in I' \cap B_{k'}$, where $|p| \le k'\alpha'/2 - M'$, we have, denoting the matrix element of $\Psi^{k'}(\hat{F})$ by $K_{\hat{F}}(\cdot, \cdot)$, 
\begin{align*}
    \langle \psi^k_{y_0 + \frac{p}{k'}},  \Phi^{k'}(\hat{F}) v' \rangle 
    &= \sum_{|m| \le k' \alpha' / 2} K_{F'}\left(y_0 + \frac{p}{k'}, y_0 + \frac{m}{k'}\right)e^{- \sqrt{-1} \langle m,  \theta_0 \rangle} \\
    &= \sum_{|m| \le k' \alpha' / 2} \hat{f}_{p-m}(y_0)e^{- \sqrt{-1} \langle m,  \theta_0 \rangle} \\
    &= \sum_{m \in \Z} \hat{f}_{p-m}(y_0)e^{- \sqrt{-1} \langle m,  \theta_0 \rangle} \\
    &= e^{- \sqrt{-1} \langle p,  \theta_0 \rangle}\hat{f}(y_0, \theta_0). 
\end{align*}
Here the third equality used the fact that, by \eqref{eq_def_hatf.y}, we have $\hat{f}_{p-m} = 0$ for $|p| \le k'\alpha'/2 - M'$ and $|m| \le k'\alpha/2$.
So we get
\begin{align*}
    \|P_{I'}\Phi^{k'}(\hat{F}) v' \|^2 = |\hat{f}(y_0, \theta_0)|^2 \cdot \#(I' \cap B_k). 
\end{align*}
Since $P_{J'}v' = v'$, we see that
\begin{align*}
    \|P_{I'}\Phi^{k'}(\hat{F}) P_{J'}\|^2 &\ge \frac{\|P_{I'}\Phi^{k'}(\hat{F}) v' \|^2}{\|v'\|^2} \\
    &\ge |\hat{f}(y_0, \theta_0)|^2 \frac{\#(I' \cap B_k)}{\#(J' \cap B_k)} \\
    &\ge \left( \|f\|_{C^0} - \frac{\delta}{50}\right)^2 
    \left( 1 - \frac{2M'}{k'\alpha'}\right) \\
    &\ge \left( \|f\|_{C^0} - \frac{\delta}{50}\right)^2 
    \left( 1 - \frac{\delta}{50 \|f\|_{C^0}}\right) \\
    & \ge \left(\|f\|_{C^0} - \frac{\delta}{25}\right)^2. 
\end{align*}
where the third inequality used \eqref{eq_max_value.y} and \eqref{eq_diff_inf.y}, the fourth inequality used \eqref{eq_def_k_model_inf.y}, and the last inequality used $0 < \delta < \|f\|_{C^0}$. 
From this and \eqref{eq_approx_const_model_inf.y}, we get
\begin{align*}
    \|\phi^{k'}(\hat{f})\| \ge \|P_{I'}\phi^{k'}(\hat{f}) P_{J'}\|
    \ge \|f\|_{C^0} - \frac{\delta}{25} - \frac{\delta}{100} \ge \|f\|_{C^0} - \frac{\delta}{10}. 
\end{align*}
This contradicts with \eqref{eq_def_k_model_inf.y}, so we get
\begin{align}\label{eq_proof_model1_inf.y}
    \inf_{k \to \infty} \|\phi^k(f)\| \ge \|f\|_{C^0}. 
\end{align}
Combining \eqref{eq_proof_model1_sup.y} and \eqref{eq_proof_model1_inf.y}, we get the desired result. 
\end{proof}

\subsection{A proof of Theorem \ref{thm_main_v2.y}}
In this subsection, we give a proof of Theorem \ref{thm_main_v2.y}. We work in the settings in subsection \ref{subsec_construction_general.y}, and use notations there. 
\begin{proof}[Proof of Theorem \ref{thm_main_v2.y}]
As a preperation, we give a sufficient condition for a family of operators on $\{\mathcal{H}_k\}_k$ to be uniformly bounded.
For a linear operator $F$ on $\mathcal{H}_k$ and a subset $V \subset B$, we write
\begin{align*}
    \mathrm{supp}(F) \subset V
\end{align*}
if $P_V F P_V = F$. 

\begin{lem}\label{lem_op_bdd.y}
Suppose we are given a family of linear operators $\{\mathcal{A}_k\}_{k \in \N}$, $\mathcal{A}_k \colon \mathcal{H}_k \to \mathcal{H}_k$ and a compact subset $B' \subset B$ such that $\mathrm{supp}(\mathcal{A}_k) \subset B'$ for all $k$. 
Suppose that we are given a finite subset $\{U_i\}_{i \in I} \subset \mathcal{U}$, a partition $B' = \sqcup_{i \in I}V_i $ and a positive constant $M > 0$ such that
\begin{enumerate}
    \item[(A)] For each $i \in I$, we have $V_i \subset U_i$. 
    \item[(B)] For each $i \in I$, $B_i(V_i, M) := \{y \in \R^n \ | x \in V_i, \ |y - x| \le M \} \subset U_i$. 
    Here we regard $U_i$ as a subset in $\R^n$ by the action coordinate, and the norm is the Euclidean norm with respect to the action coordinate on $U_i$.
\end{enumerate}

Furthermore, we assume that,   
\begin{enumerate}
    \item For each $i \in I$, there exists a constant $C_i > 0$ such that, for all $k\in \N$, $x \in V_i$ and $p \in \Z^n$ with $|p| \le kM$, we have
    \begin{align*}
        \|P_{x + p/k} \mathcal{A}_k P_x\| \le \frac{C_i}{(1 + |p|)^{n + 1}}. 
    \end{align*}
    \label{cond_bd_1.y}
    \item There exists a constant $C' > 0$ such that, for all $k \in \N$ and pairs $(b, c) \in B_k \times B_k$ which cannot be expressed as $(b, c ) = (x + p/k, x)$ with $|p| \le kM$ in the action coordinate on $U_i$ with $c \in V_i$, we have
    \begin{align*}
        \|P_b \mathcal{A}_k P_c\| \le \frac{C'}{k^{2n}}. 
    \end{align*}
    \label{cond_bd_2.y}
\end{enumerate}
Then we have
\begin{align*}
    \sup_{k \to \infty}\|\mathcal{A}_k\| < \infty. 
\end{align*}
\end{lem}

\begin{proof}
First we note that, since $B'$ is compact, there exists a positive constant $a < + \infty$ such that for all $k \in \N$, 
\begin{align}\label{ineq_finite_lattice.y}
    \# (B_k \cap B') < \frac{a}{k^n}. 
\end{align}
Fix $k \in \N$. Since $\mathrm{supp}(\mathcal{A}_k) \subset \sqcup_{i\in I}V_i$, we can decompose the operator $\mathcal{A}_k$ as
\begin{align}\label{eq_A_k.y}
    \mathcal{A}_k = \sum_{i \in I} P_{U_i} \mathcal{A}_k P_{V_i} + \sum_{i \in I} (1 - P_{U_i}) \mathcal{A}_k P_{V_i}. 
\end{align}
First we estimate the first term in \eqref{eq_A_k.y}. 
We have
\begin{align*}
    P_{U_i} \mathcal{A}_k P_{V_i}  = \sum_{p \in \Z^n, |p| \le kM} \sum_{x \in V_i \cap B_k} P_{x + p/k}\mathcal{A}_k P_x
    + \sum_{x \in V_i \cap B_k} P_{U_i \setminus B_i(x, M)} \mathcal{A}_k P_{x}. 
\end{align*}
For each $p \in \Z^n$ with $|p| \le kM$, we have (see Lemma \ref{lem_est_norm.y})
\begin{align*}
    \|\sum_{x \in V_i \cap B_k} P_{x + p/k}\mathcal{A}_k P_x\| = 
    \max_{x \in V_i \cap B_k} \|P_{x + p/k} \mathcal{A}_k P_x\|
    \le \frac{C_i}{(1 + |p|)^{n + 1}}, 
\end{align*}
by the condition \eqref{cond_bd_1.y} in the statement. 
So we have
\begin{align*}
   \| \sum_{p \in \Z^n, |p| \le kM} \sum_{x \in V_i \cap B_k} P_{x + p/k}\mathcal{A}_k P_x\|
   &\le C_i \sum_{p \in \Z^n, |p| \le kM}\frac{1}{(1 + |p|)^{n + 1}} \\
   &\le C_i \sum_{p \in \Z^n}\frac{1}{(1 + |p|)^{n + 1}}
   =C_iC'', 
\end{align*}
for some $C'' < \infty$ which does not depend on $k$ nor $i\in I$. 
Moreover we have, for each $x \in V_i$, using the condition \eqref{cond_bd_2.y} in the statement, 
\begin{align*}
    \|P_{U_i \setminus B_i(x, M)} \mathcal{A}_k P_{x}\| \le \# (B_k\cap B) \cdot \frac{C'}{k^{2n}}. 
\end{align*}
Combining these, for each $i$, we have
\begin{align*}
    \|P_{U_i} \mathcal{A}_k P_{V_i}\| \le C_iC'' + (\# (B_k \cap B))^2 \cdot \frac{C'}{k^{2n}}.
\end{align*}
Next we estimate the second term in \eqref{eq_A_k.y}. 
Since all the pair $(b, c)$ with $c \in V_i$ and $b \notin U_i$ satisfies the assumption for the condition \eqref{cond_bd_2.y} in the statement of the Lemma, we get
\begin{align*}
    \|\sum_{i \in I} (1 - P_{U_i}) \mathcal{A}_k P_{V_i}\|
    \le (\# (B_k \cap B))^2 \cdot \frac{C'}{k^{2n}}. 
\end{align*}
Combining these, we get
\begin{align*}
    \|\mathcal{A}_k\| \le C''\sum_{i \in I}C_i +( \# I + 1 ) \cdot (\# B_k)^2 \cdot \frac{C'}{k^{2n}}
    \le  C''\sum_{i \in I}C_i +( \# I + 1 ) \cdot a^2 C'
\end{align*}
The last inequality used \eqref{ineq_finite_lattice.y}. 
Since we have $\# I < \infty$, we get the result. 
\end{proof}

Let us fix $f, g \in C_c^\infty(X)$ and let $\mathcal{A}_{err}^k$ be the operator on $\mathcal{H}_k$ defined by
\begin{align}\label{eq_A_err.y}
    \mathcal{A}_{err}^k := k^2 \left(\phi^k(f)\phi^k(g) - \phi^k(fg) + \frac{\sqrt{-1}}{2k}\phi^k(\{f, g\})\right). 
\end{align}
It is enough to show that the family of operators $\{ \mathcal{A}_{err}^k\}_k$ satisfies the conditions in Lemma \ref{lem_op_bdd.y}. 

Since $f, g$ are compactly supported and the covering $\mathcal{U}$ is locally finite and consists of relatively compact subsets, it is easy to see that there exists a compact subset $B' \subset B$ such that for all $k \in \N$, we have 
\begin{align*}
    \mathrm{supp}(\mathcal{A}_{err}^k), \ \mathrm{supp}(\phi^k(f)),\ \mathrm{supp}(\phi^k(g)) \subset B'. 
\end{align*}
We fix such $B'$. 
Then we fix a finite subset $\{U_i\}_{i \in I} \subset \mathcal{U}$, a partition a partition $B' = \sqcup_{i \in I}V_i $ and a positive constant $M > 0$ satisfying the conditions (A) and (B) in Lemma \ref{lem_op_bdd.y} (such datum always exist). 
\begin{rem}\label{rem_D_0.y}
There exist a positive constant $D_0 > 0$ such that, for any $i, j \in I$, any positive number $r > 0$ and any pair $(b, c) \in B \times B$ with $c \in U_i\cap U_j$ such that $B_{i}(c, r) \subset U_i$, we have
\begin{align*}
    b \notin B_i(c, r) \Rightarrow b \notin B_j(c, D_0r).  
\end{align*}
Here $B_i$, $B_j$ denotes the Euclidean ball with respect to the action coordinates on $U_i$, $U_j$, respectively. 
This can be seen as follows. 
Denoting the transition function of the coordinates on $U_j$ and $U_i$ by
\begin{align*}
    x \mapsto B_{ij}x + c_{ij}, 
\end{align*}
it is enough to set
\begin{align*}
    D_0 := \left( \max_{i, j \mbox{ with }U_i \cap U_j \neq \phi} \|B_{ij}\| \right)^{-1}. 
\end{align*}
From now on we fix such $D_0$. 
\end{rem}

Since $f, g$ are smooth and compactly supported, 
for each $N \in \N$, we have a positive constant $C_N > 0$ with the following conditions. 
 For all $i \in I$, if we express the Fourier coefficients of $f$ on $U_i$ with respect to the action-angle coordinates by $f_m(x)$, $m \in \Z^n$, and similarly for other functions, we have
    \begin{align}
        \|f_m\|_{U_i}, \|\nabla f_m\|_{U_i}, \|g_m\|_{U_i}, \|\nabla g_m\|_{U_i}, \|\{f, g\}_m\|_{U_i}, \|(fg)_m\|_{U_i} \le \frac{C_N}{(1 + |m|)^N},  \label{eq_bd_fourier.y}\\
        \|H( f_m)\|_{U_i}, \|H(g_m)\|_{U_i} \le \frac{C_N}{(1 + |m|)^N}.\label{eq_bd_hess.y}
    \end{align}
    Here we denote by $H(\cdot)$ the Hessian of a function, and we abuse the notations to write $\|\cdot\|_{U_i}$ the $C^0(U_i)$-norms with respect to the flat metrics induced by the Euclidean metric of the action coordinate on $U_i$. 

First we check the condition \eqref{cond_bd_2.y} in Lemma \ref{lem_op_bdd.y}. 
Let us take $(b,c ) \in B_k\times B_k$ which satisfies the assumption in the condition \eqref{cond_bd_2.y}. 
By Remark \ref{rem_D_0.y}, we have $b \notin B_i(c, D_0M)$ for all $i \in I$ with $c \in U_i$. 
By Lemma \ref{lem_arg.y} and \eqref{eq_bd_fourier.y}, we have
\begin{align}\label{eq_thm_est_1.y}
    \|P_b\phi^k(fg) P_c\| &\le \frac{C_N}{(1 + kD_0M)^N}, \\
    \|P_b\phi^k(\{f, g\}) P_c\| &\le \frac{C_N}{(1 + kD_0M)^N}. \notag
\end{align}
Also we have
\begin{align*}
    P_b\phi^k(f)\phi^k(g) P_c 
    =\sum_{d \in B_k}(P_b\phi^k(f) P_d)( P_d \phi^k(g) P_c). 
\end{align*}
The terms in the sum of the last equation is nonzero only when $(b, d)$ and $(d, c)$ are both close. 
Let us take $i \in I$ so that $c \in V_i$. 
Then by the assumption, either
\begin{enumerate}
    \item[(a)] $d \notin B_i(c, M/2)$, or
    \item[(b)] $d \in B_i(c, M/2)$, $B_i(d, M/2) \subset U_i$ and $b \notin B_i(d, M/2)$
\end{enumerate}
holds. 
In the case (a), we have $\|P_d \phi^k(g) P_c\| \le \frac{C_N}{(1 + kD_0M/2)^N}$ by Lemma \ref{lem_arg.y} and \eqref{eq_bd_fourier.y}, and we also have $\|P_b \phi^k(f) P_d\| \le C_1$. 
In the case (b), we have $\|P_b \phi^k(f) P_d\|\le \frac{C_N}{(1 + kD_0M/2)^N}$ and $\|P_d \phi^k(g) P_c\| \le C_1$ similarly. 
Thus in both cases we get
\begin{align*}
    \|(P_b\phi^k(f) P_d)( P_d \phi^k(g) P_c)\| \le \frac{C_1C_N}{(1 + kD_0M/2)^N}. 
\end{align*}
So we have
\begin{align}\label{eq_thm_est_2.y}
    \|P_b\phi^k(f)\phi^k(g) P_c\| \le \# (B_k\cap B') \cdot \frac{C_1C_N}{(1 + kD_0M/2)^N}. 
\end{align}
Since $B'$ is compact, $\# (B_k\cap B') = O(k^{n})$ as $k \to \infty$, the right hand side above is $O(k^{-N+n})$. 

Combining \eqref{eq_thm_est_1.y} and \eqref{eq_thm_est_2.y}, we see that $\|P_b A^k_{err} P_c\|$ is of $O(k^{-N})$ for any $N \in \N$ (with coefficients independent of $b$ or $c$), so in particular we have a constant $\tilde{C} > 0$ which is independent of $k$, such that
\begin{align}\label{eq_thm_est_offdiag.y}
    \|P_b \mathcal{A}^k_{err}P_c\| \le \frac{\tilde{C}}{k^{2n}}, 
\end{align}
for all $(b, c) \in B_k \times B_k$ satisfying the assumption in \eqref{cond_bd_2.y} of Lemma \ref{lem_op_bdd.y}. 

Next, we check the condition \eqref{cond_bd_1.y} in Lemma \ref{lem_op_bdd.y}. 
Only in this proof, we use the following notation. 
Define the set $W_i$ for $i \in I$ by
\begin{align*}
    W_i := \left\{(x, p, k) \in \left(V_i \cap \frac{\Z^n}{k}\right) \times \Z^n \times \N  : |p| \le kM\right\}. 
\end{align*}
The pairs $(b, c) \in B_k \times B_k$ satisfying the condition \eqref{cond_bd_1.y} in Lemma \ref{lem_op_bdd.y} are precisely those which can be expressed as $(x + \frac{p}{k}, x)$ for an element $(x, p, k) \in W_i$ for some $i \in I$. 
Until the end of this proof, we only consider elements $(x, p, k)$ belonging to $W_i$ for some $i \in I$. 

Let us fix $i \in I$.  
As in Lemma \ref{lem_arg.y}, using the action-angle coordinate on $X_{U_i}$ and the trivialization of $(L, \nabla)|_{X_{U_i}}$, we take the orthonormal basis of $\mathcal{H}_k|_{U_i}$, and express operators on $\mathcal{H}_k$ by the matrix coefficients with respect to that basis. 
For any $(x, p, k) \in W_i$, we have
\begin{align*}
    K_{fg}\left(x + \frac{p}{k}, x\right) = \exp\left(\sqrt{-1}C(x, p,k) \frac{|p|^3}{k^2}\right)\sum_{m \in \Z^n}(f_{p-m}g_{m})|_{x + \frac{p}{2k}}. 
\end{align*}
Here $C(x, p, k)$ is a real number such that $|C(x, p, k)| \le \|\nabla A\|_{U_i}$ by Lemma \ref{lem_arg.y}. 
First we observe that, for any $N \in \N$, there exists a constant $C'_N > 0$ such that, for all $(x, p, k) \in W_i$,
\begin{align*}
    \left\|K_{fg}\left(x + \frac{p}{k}, x\right) - \exp\left(\sqrt{-1}C(x, p, k) \frac{|p|^3}{k^2}\right)\sum_{m \in \Z^n, |m| \le kM}(f_{p-m}g_{m})|_{x + \frac{p}{2k}}\right\| \le \frac{C'_N}{k^N}. 
\end{align*}
i.e., we may replace the sum over $m \in \Z^n$ with those with $|m| \le kM $. 
Indeed, the above difference is bounded by
\begin{align*}
    \sum_{m \in \Z^n, |m| > kM} \|f_{p-m}\|_{U_i}\cdot \|g_m\|_{U_i}
    \le \sum_{m \in \Z^n, |m| > kM} \frac{C_1C_N}{(1 + |m|)^N}. 
\end{align*}
by \eqref{eq_bd_fourier.y}, and this is $O(k^{-N+n})$. 

Put $1 + \tilde{C}_{x, p, k}\frac{|p|^3}{k^2} = \exp(\sqrt{-1}C(x, p, k) \frac{|p|^3}{k^2})$. 
We have $|\tilde{C}_{x, p, k}| \le \|\nabla A\|_{U_i}$. 
We have, for all $(x, p, k) \in W_i$, 
\begin{align*}
    \left\|K_{fg}\left(x + \frac{p}{k}, x\right) - \left(1 + \tilde{C}_{x, p, k}\frac{|p|^3}{k^2} \right) \sum_{m \in \Z^n, |m| \le kM}(f_{p-m}g_{m})|_{x + \frac{p}{2k}}\right\| \le \frac{C'_N}{k^N}. 
\end{align*}
Moreover, we have, using the inequality $(1 + |p-m|)(1 + |m|) \ge 1 + |p|$, for any $N \in \N$, $p \in \Z^n $ and $k \in \N$, we have
\begin{align}\label{eq_sum_tech.y}
    \sum_{m \in \Z^n, |m| \le kM}\|f_{p-m}\|_{U_i} \cdot \|g_{m}\|_{U_i}
    &\le \sum_{m \in \Z^n, |m| \le kM} \frac{C_NC_{N + n + 1}}{(1 + |p-m|)^N(1 + |m|)^{N + n + 1}} \\ 
    &\le \frac{C_NC_{N+n+1}}{(1 + |p|)^N} \sum_{m \in \Z^n}\frac{1}{(1 + |m|)^{n+1}} \notag\\
    &\le \frac{D'_N}{(1 + |p|)^N}, \notag
\end{align}
where $D'_N > 0$ is a positive constant independent of $k$, $p$.
Thus we get, for any $N \in \N$ and $(x, p, k) \in W_i$, 
\begin{align}\label{eq_thm_est_fg.y}
    \left\|K_{fg}\left(x + \frac{p}{k}, x\right) - \sum_{m \in \Z^n, |m| \le kM}(f_{p-m}g_{m})|_{x + \frac{p}{2k}}\right\| &\le
    \|\nabla A\|_{U_i}\frac{|p|^3}{k^2}\cdot\frac{D'_{N + 3}}{(1 + |p|)^{N+3}} + \frac{C'_{N+2}}{k^{N+2}}\\
    &\le \frac{1}{k^2}\frac{\|\nabla A\|_{U_i}D'_{N + 3} +(M+1)^{N}C'_{N+2} }{(1 + |p|)^N} 
    . \notag
\end{align}
where the last inequality uses the fact that $|p|+1 \le kM +1\le k(M+1) $. 
We apply the same argument for $\{f, g\}$ and get that, there exists a constant $D''_N$ such that for all $(x, p, k) \in W_i$, we have
\begin{align}\label{eq_thm_est_bracket.y}
    &\left\| K_{\{f, g\}}\left(x +\frac{p}{k}, x\right) -
       \sqrt{-1}\sum_{m \in \Z^n, |m| \le kM}\left.\left\{ \langle m, \nabla f_{p-m} \rangle g_{m}
       - f_{p-m}\langle p-m, \nabla g_{m} \rangle  \right\}\right|_{x+ \frac{p}{2k}}\right\| \\
     &\le \frac{1}{k^2} \frac{D''_N}{(1 + |p|)^N}.  \notag
\end{align}

We have, for $(x, p, k) \in W_i$, 
\begin{align*}
    P_{x + p/k}\phi^k(f)\phi^k(g) P_x 
    =\sum_{d \in B_k}(P_{x + p/k}\phi^k(f) P_d)( P_d \phi^k(g) P_x). 
\end{align*}
When a point $d \in B_k$ satisfies $d \notin B_i(x, M)$, we have 
\begin{align*}
    \|(P_{x + p/k}\phi^k(f) P_d)( P_d \phi^k(g) P_x)\| \le \frac{C_1C_N}{(1 + kD_0M)^N}, 
\end{align*}
for any $N$, similarly as before. 
Since we have $\mathrm{supp}(\phi^k(g)) \subset B'$ for the compact subset $B' \subset B$ and we have $\# (B' \cap B_k) = O(k^n)$, we see that for each $N$, there exists a constant $C_N''$ such that, for any $(x, p, k) \in W_i$, we have
\begin{align*}
\left\|K_{\phi^k(f)\phi^k(g)}\left(x + \frac{p}{k}, x\right) 
    - \sum_{m \in \Z^n, |m| \le kM} K_f\left(x + \frac{p}{k}, x + \frac{m}{k}\right)K_g\left(x + \frac{m}{k}, x\right) \right\| \le \frac{C''_N}{k^N}. 
\end{align*}
Moreover we have, by Lemma \ref{lem_arg.y}, for any $(x, p, k) \in W_i$ and $m \in \Z^n$ with $|m| \le kM$,  
\begin{align*}
    &\left|K_f\left(x + \frac{p}{k}, x + \frac{m}{k}\right)K_g\left(x + \frac{m}{k}, x\right)
    -f_{p-m}\left(x + \frac{p + m}{2k}\right)g_m\left(x + \frac{m}{2k}\right)\right| \\
    &\le \frac{|p-m|^3 + |m|^3}{k^2}\cdot \|\nabla A\|_{U_i} \cdot \|f_{p-m}\|_{U_i}\cdot \|g_m\|_{U_i}. 
\end{align*}
By the same argument as \eqref{eq_sum_tech.y}, we have
\begin{align*}
\sum_{m \in \Z^n, |m| \le kM}(|p-m|^3 + |m|^3)\cdot \|f_{p-m}\|_{U_i}\cdot \|g_m\|_{U_i}
\le \frac{C'''_N}{(1 + |p|)^N}
\end{align*}
for some constant $C'''_N$ independent of $k$.
Thus we see that, for each $N$ there exists a constant $D'''_N$ such that, for any $(x, p, k) \in W_i$, 
\begin{align}\label{eq_thm_est_comp.y}
    \left\|K_{\phi^k(f)\phi^k(g)}\left(x + \frac{p}{k}, x\right) 
    - \sum_{m \in \Z^n, |m| \le kM} f_{p-m}\left(x + \frac{p +m}{2k}\right)g_m\left(x + \frac{m}{2k}\right) \right\| \le \frac{1}{k^2}\frac{D'''_N}{(1 + |p|)^N}. 
\end{align}
Moreover we have, by \eqref{eq_bd_hess.y}, 
\begin{align*}
    &\left\|f_{p-m}\left(x + \frac{m+ p}{2k}\right)
    -\left.\left\{f_{p-m} + \left\langle \frac{m}{2k}, \nabla f_{p-m}\right\rangle\right\}\right|_{x + \frac{p}{2k}}\right\| \le \frac{C_N}{(1 + |p-m|)^N}\frac{|m|^2}{8k^2}, \\
    &\left\|g_m\left(x + \frac{m}{2k}\right)
    -\left.\left\{ g_m - \left\langle \frac{p-m}{2k}, \nabla g_m\right\rangle\right\}\right|_{x + \frac{p}{2k}}\right\| \le \frac{C_N}{(1 + |m|)^N}\frac{|p-m|^2}{8k^2}. 
\end{align*}
We estimate
\begin{align*}
    &\frac{C_N}{(1 + |p-m|)^N}\frac{|m|^2}{8k^2}\cdot \|g_m\| 
    \le \frac{C_NC_{N+n+3}}{8k^2(1 + |p - m|)^N(1 + |m|)^{N + n + 1}} , \\
    &\frac{C_{N + n + 1}}{(1 + |m|)^{N+n+1}}\frac{|p-m|^2}{8k^2} \cdot \|f_{p-m}\|\le \frac{C_{N+2}C_{N+n+1}}{8k^2(1 + |p - m|)^N(1 + |m|)^{N + n + 1}}, \\
    &\|\langle \frac{m}{2k}, \nabla f_{p-m}\rangle\| \cdot \|\langle \frac{p-m}{2k}, \nabla g_m\rangle\|
    \le  \frac{C_{N+1}C_{N+n+2}}{4k^2(1 + |p - m|)^N(1 + |m|)^{N + n + 1}}. 
\end{align*}
So by a similar estimate as in \eqref{eq_sum_tech.y}, we see that, for each $N$ there exists a constant $D''''_N$ such that, for any $(x, p, k) \in W_i$ we have
\begin{align}\label{eq_thm_est_final.y}
    &\left\|\sum_{m \in \Z^n, |m| \le kM} f_{p-m}\left(x + \frac{p +m}{2k}\right)g_m\left(x + \frac{m}{2k}\right) \right. \\
    &\left. -\sum_{m \in \Z^n, |m| \le kM}\left.\left\{ f_{p-m}\cdot g_{m} + \left\langle \frac{m}{2k}, \nabla f_{p-m}\right\rangle g_m
    - f_{p-m}\cdot\left\langle \frac{p-m}{2k}, \nabla g_m\right\rangle \right\}\right|_{x + \frac{p}{2k}}\right\|\notag \\
    & \le \frac{1}{k^2}\cdot \frac{D''''_N}{(1 + |p|)^N}.  \notag
\end{align}

Combining \eqref{eq_thm_est_fg.y}, \eqref{eq_thm_est_bracket.y}, \eqref{eq_thm_est_comp.y} and \eqref{eq_thm_est_final.y}, there exists a constant $\tilde{D}_i$ such that, for any $(x, p, k) \in W_i$ we have
\begin{align}\label{eq_thm_diag.y}
    \|P_{x + p/k}\mathcal{A}^k_{err}P_x\| \le \frac{\tilde{D}_i}{(1 + |p|)^{n + 1}}. 
\end{align}

By \eqref{eq_thm_est_offdiag.y} and \eqref{eq_thm_diag.y}, we see that the family of operators $\{\mathcal{A}_{err}^k\}_k$ satisfies the conditions of Lemma \ref{lem_op_bdd.y}, so we get the result. 
\end{proof}

\section*{Acknowledgment}
The author is grateful to Mikio Furuta, Hajime Fujita, Kota Hattori, Kaoru Ono and Takahiko Yoshida for interesting discussions. 
This work is supported by Grant-in-Aid for JSPS Fellows Grant Number 19J22404.

\bibliographystyle{plain}
\bibliography{quantization}

\end{document}